\documentclass[final]{elsarticle}
\usepackage[utf8]{inputenc}
\usepackage{url}
\usepackage{color,comment}
\usepackage{ulem}
\usepackage{graphicx}
\usepackage{amsmath,amssymb}
\usepackage{float}
\usepackage{graphicx}
\usepackage{appendix}
\usepackage{listings}
\usepackage{subfigure}
\usepackage{bm}
\usepackage{makecell}
\usepackage{multirow}
\usepackage[T1]{fontenc}
\usepackage[utf8]{inputenc}
\usepackage{algorithm,algorithmicx,algpseudocode}
\usepackage{multicol}
\usepackage[pdftex, unicode=true,
            colorlinks,
            linkcolor=red,
            anchorcolor=blue,
            citecolor=green
            ]{hyperref}

\newtheorem{theorem}{Theorem}[section]

\newtheorem{remark}{Remark}[section]

\newtheorem{assumption}{Assumption}[section]
\newenvironment{proof}{\paragraph{Proof}}{\hfill$\square$}
\makeatletter
\newenvironment{breakablealgorithm}
  {\begin{center}
     \refstepcounter{algorithm}
     \hrule height.8pt depth0pt \kern2pt
     \renewcommand{\caption}[2][\relax]{
       {\raggedright\textbf{\ALG@name~\thealgorithm} ##2\par}
       \ifx\relax##1\relax
         \addcontentsline{loa}{algorithm}{\protect\numberline{\thealgorithm}##2}
       \else
         \addcontentsline{loa}{algorithm}{\protect\numberline{\thealgorithm}##1}
       \fi
       \kern2pt\hrule\kern2pt
     }
  }{\kern2pt\hrule\relax
   \end{center}
  }
\makeatother

\newcommand{\be}{\bm e}

\newcommand{\bN}{\bm N}
\newcommand{\bff}{\bm f}
\newcommand{\bn}{\bm n}
\newcommand{\bw}{\bm w}

\newcommand{\bX}{\bm X}
\newcommand{\bx}{\bm x}

\newcommand{\bu}{\bm u}

\newcommand{\bv}{\bm v}

\newcommand{\bg}{\bm g}

\newcommand{\bS}{\bm \sigma}
\newcommand{\bD}{\bm D}

\begin{document}
\begin{frontmatter}

\title{Physics-informed neural networks for solving two-phase flow problems with moving interfaces}

\author[chuanda]{Qijia Zhai}
\ead{zhaiqijia@stu.scu.edu.cn}
\address[chuanda]{School of Mathematics, Sichuan University, Chengdu 610064, China}

\author[unlv]{Pengtao Sun\corref{cor}}
\ead{pengtao.sun@unlv.edu}
\address[unlv]{Department of Mathematical Sciences, University of Nevada Las
Vegas, Las Vegas, Nevada 89154, USA}

\author[chuanda,xie]{Xiaoping Xie}
\ead{xpxie@scu.edu.cn}
\address[xie]{National Key Laboratory of Fundamental Algorithms and Models for Engineering Simulation, Sichuan University, Chengdu 610207, China}

\author[dali]{Xingwen Zhu}
\ead{zxw4688@126.com}
\address[dali]{School of Mathematics and Computer Sciences, Dali University, 2 Hongsheng Street, Dali, Yunnan
671003, China}

\author[LSEC]{Chen-Song Zhang}
\ead{zhangcs@lsec.cc.ac.cn}
\address[LSEC]{LSEC \& NCMIS, Academy of Mathematics and System Science, Beijing, China.}

\cortext[cor]{Corresponding author}

\begin{abstract}
In this paper, a meshfree method using physics-informed neural
networks (PINNs) is developed for solving two-phase flow problems
with moving interfaces, where two immiscible fluids bearing
different material properties, are separated by a dynamically
evolving interface and interact with each other through interface
conditions.
Two kinds of distinct scenarios of interface motion are addressed:
the prescribed interface motion whose moving velocity is explicitly
given, and the solution-driven interface motion whose evolution is
determined by the velocity field of two-phase flow. Based upon
piecewise deep neural networks and spatiotemporal sampling
points/training set in each fluid subdomain, the proposed PINNs
framework reformulates the two-phase flow moving interface problem
as a least-squares (LS) minimization problem, which involves all
residuals of governing equations, interface conditions, boundary
conditions and initial conditions. Furthermore, approximation
properties of the proposed PINNs approach are analyzed rigorously
for the presented two-phase flow model by employing the Reynolds
transport theorem in evolving domains,
moreover, a comprehensive error estimation is provided to account
for additional complexities introduced by the moving interface and
the coupling between fluid dynamics
and interface evolution. 
Numerical experiments are carried out to illustrate the
effectiveness of the proposed PINNs approach for various
configurations of two-phase flow moving
interface problems, 
and to validate the theoretical findings as well.
A practical guidance is thus provided for an efficient training set
distribution when applying the proposed PINNs approach to two-phase
flow moving interface problems in practice.
\end{abstract}


\end{frontmatter}

\section{Introduction}
Two-phase flow problems with moving interfaces are fundamental in
numerous scientific and engineering applications, ranging from
microfluidics and biomedical engineering to environmental fluid
dynamics and industrial processes (e.g., see
\cite{Jacqmin1999,cerne2001coupling,FengLiu2007,quan2007moving,mirjalili2017interface}
and references therein). These problems involve the interaction
between two immiscible fluids with different material properties
across a dynamically evolving interface, where the interface
position and shape change over time according to either prescribed
movement patterns or solution-driven evolution governed by fluid
dynamics. Usually, the mathematical modeling of such problems
requires solving the incompressible Navier-Stokes equations in each
fluid domain, coupled through interface conditions that ensure
continuity of velocity (kinematic condition) and balance of forces
(dynamic condition) across the moving interface. As is well known,
the presence of moving interfaces introduces not only significant
computational challenges but also particular theoretical barriers
when applying energy error estimates in evolving domains.

In specific, the complexity of two-phase flow moving interface
problems stems from the following several key factors. First, the
interface geometry evolves continuously in time, requiring
sophisticated tracking mechanisms and adaptive numerical strategies.
Second, the material properties exhibit sharp discontinuities across
the interface, leading to singularities in the solution and
requiring careful handling of jump conditions. Third, the coupling
between two fluid phases through interface conditions creates
additional nonlinearity and complexity in the governing equations.
Fourth, the time-dependent nature of fluid domains that adapts to
the moving interface introduces boundary integral terms over the
moving interface in energy error estimates, which must be properly
handled to ensure theoretical rigor and numerical
stability~\cite{quan2007moving,cerne2001coupling}. Fifth, for the
case of solution-driven interface motion, the interface evolution is
determined by the fluid velocity field through kinematic coupling
conditions, creating a complex bidirectional coupling between fluid
dynamics and interface evolution that significantly complicates both
theoretical analysis and numerical implementation.

Traditional numerical methods for solving two-phase flow problems
can be broadly categorized into mesh-based methods and meshfree
methods. Mesh-based methods include interface-fitted approaches such
as classical finite element
methods~\cite{babuvska1970finite,bramble1996finite,Chen.Z;Zou.J1988a},
discontinuous Galerkin
methods~\cite{massjung2012unfitted,cai2011discontinuous}, and
virtual element methods~\cite{chen2017interface}, as well as
interface-unfitted approaches such as immersed boundary
methods~\cite{peskin_2002}, immersed interface
methods~\cite{doi:10.1137/0731054,Li.Z;Lai.M2001a}, immersed finite
volume
methods~\cite{OEVERMANN20095184,10.1007/978-3-540-78827-0_78},
matched interface and boundary methods~\cite{ZHOU20061}, ghost fluid
methods~\cite{liu2000boundary}, extended finite element
methods~\cite{fries2010extended}, cut finite element
methods~\cite{burman2015cutfem}, and immersed finite element
methods~\cite{doi:10.1137/130912700}. While mesh-based methods have
achieved considerable success in solving interface problems, they
face significant challenges when dealing with moving interfaces,
where numerical implementations become particularly complex due to
the need for frequent mesh adaptation, interface tracking, and
handling of jump conditions across evolving boundaries. The
computational cost associated with mesh generation and adaptation
for complex moving interfaces can be prohibitive, especially for
three-dimensional problems with intricate interface geometries. For
the case of solution-driven interface motion, the additional
complexity of coupling interface evolution with fluid dynamics
further exacerbates these challenges.

Meshfree methods offer an attractive alternative by circumventing
the need for explicit mesh generation and providing natural handling
of complex geometries. These methods have been successfully applied
to various realistic problems including solid
mechanics~\cite{liu2016overview}, fracture
mechanics~\cite{daxini2014review}, and fluid-structure
interaction~\cite{mazhar2021meshfree,hu2019consistent,hu2019spatially}.
A comprehensive review of meshfree method and its applications can
be found in~\cite{garg2018meshfree}. Recently, deep neural networks
(DNNs) have emerged as a powerful tool for solving partial
differential equations (PDEs)
\cite{CaiChenLiuLiu2019,Dissanayake1994,EYu,HeLiXuZheng},
particularly for challenging problems that are difficult to handle
with traditional numerical
methods~\cite{HanJentzenE,SamaniegoGoswami,SirignanoSpilopoulos,TranHamiltonMckayQuiringVassilevski,WangZhang}.
Among these DNN approaches, physics-informed neural networks (PINNs)
have shown particular potential in solving interface problems due to
their flexibilities, where the piecewise DNNs are utilized to
approximate solutions in different subdomains while interface
conditions are treated together with the governing equations in the
least-squares (LS) minimization problem
\cite{HeLinHu2022,ZhuHuSun2023,SarmaRoy2024,LiFan2025}.

Recent advances in applying PINNs to two-phase flow problems have
demonstrated significant potential for addressing complex interface
dynamics. For instance, a data-driven simulation of two-phase fluid
processes with heat transfer is developed in
\cite{jalili2024physics} where PINNs are applied to capture
interface behavior and model hydrodynamics and heat transfer in flow
configurations. 
PINNs for the phase field modeling (PF-PINNs) is proposed for a
two-dimensional immiscible incompressible two-phase flow
\cite{qiu2022physics} in which the Cahn-Hilliard/phase field
equation and Navier-Stokes equations are directly encoded into
neural network residuals.
A residual-based adaptive PINN framework is developed for two-phase
flow in porous media \cite{hanna2022residual}, where an adaptive
refinement strategy is introduced by capturing moving flow fronts to
improve the accuracy. 

Despite these significant advances, the existing PINN approaches for
two-phase flow problems exhibit several fundamental limitations that
restrict their applicabilities to moving interface scenarios. In
particular, the work shown in \cite{jalili2024physics} relies
heavily on pre-computed numerical results of fluid field for the
training purpose, which limits its generalizability to problems
where such data is unavailable or computationally expensive to
generate. Moreover, their work does not address theoretical
challenges associated with moving boundaries in time-dependent
domains.
The PF-PINNs method proposed in \cite{qiu2022physics} is inherently
limited to diffuse interface models and cannot directly handle sharp
interface problems with discontinuous material properties across the
interface. The adaptive PINNs framework studied in
\cite{hanna2022residual} lacks rigorous error estimates for moving
interface problems and does not provide comprehensive error bounds
that account for the additional complexity introduced by
time-dependent domains. In summary, none of these approaches provide
a rigorous theoretical analysis framework for handling PINN-based
energy error estimates in time-dependent domains adapting to moving
interfaces, which is essential to establish a theoretical
convergence guarantee and ensure physical consistency for two-phase
flow moving interface problems. Furthermore, these existing methods
primarily focus on prescribed interface motion scenarios and do not
address the additional complexity of solution-driven interface
motion, where the interface evolution is driven by the fluid
dynamics from either side of the interface through kinematic- and
dynamic interface conditions.


In this paper, we develop a comprehensive framework for two-phase
flow problems with moving interfaces using PINNs, in the sense of
both numerical methodology and theoretical analysis, to address both
prescribed interface motion and solution-driven interface motion
scenarios. The proposed PINNs method reformulates the two-phase flow
moving interface problem as a LS minimization problem that involves
residuals of Navier-Stokes equations, interface conditions, boundary
conditions and initial conditions, and utilizes piecewise deep
neural networks to approximate fluid velocity and pressure fields in
each fluid subdomain, where spatiotemporal sampling points/training
sets are defined on the evolving interfaces and interior subdomains,
as well as the fixed boundaries and initial domains. Unlike previous
approaches, our method provides the following primary advantages:
(1) zero or very few preset training data are required, making the
proposed DNN approach truly physics-informed and generalizable to a
new two-phase flow problem with moving interface; (2) the moving
sharp interface problem with discontinuous material properties is
directly handled to avoid limitations of diffuse interface model;
(3) a rigorous theoretical analysis is provided by employing the
Reynolds transport theorem to conduct energy error estimates in
evolving domains, with comprehensive error bounds that account for
moving interface effects and interface coupling; (4) a practical
guidance is offered for the optimal distribution of a spatiotemporal
sampling points/training set based on theoretical insights rather
than empirical tuning; and (5) both prescribed and solution-driven
interface motion scenarios can be handled within a unified
theoretical framework.


Specifically, our theoretical analysis reveals that the
generalization errors of the developed PINNs approach for two-phase
flow moving interface problems involve quadrature errors arising
from interface conditions across the moving interface as well, in
addition to the interior PDEs, boundary conditions and initial
conditions.
This insight is particularly valuable for practitioners, as it
suggests that computational resources should be allocated
strategically, with increased sampling density on the interface
to achieve optimal accuracy. The error bounds also provide
theoretical guarantees for the convergence of the proposed method,
ensuring that the neural network approximation converges to the true
solution as the number of sampling points increases. For the case of
solution-driven interface motion, our analysis accounts for the
additional coupling terms that arise from the kinematic interface
condition, providing theoretical guarantees even in the presence of
the complex bidirectional coupling.

Our theoretical findings are validated by numerical experiments for
various interface configurations in this paper, including
deformable, translational, rotational and solution-driven interface
evolutions, as well as high-contrast interface coupling,
which illustrates that the proposed method can handle complex moving
interfaces in two-phase flow problems without the need of mesh
generation, making it particularly attractive in practical
applications with complex or rapidly changing interface geometries.
It is especially suitable for scenarios involving interface motion
related to the solution, where traditional mesh-based methods face
significant challenges.

The rest of the paper is organized as follows. We introduce the
model of two-phase flow problems with moving interfaces in
Section~\ref{sec:model},
and develop its PINNs methodology in
Section~\ref{sec:interface-problems}, including the LS formulation,
neural network architecture, and discrete implementation. The
comprehensive error analysis framework is presented in
Section~\ref{sec:error} to provide theoretical convergence
guarantees and practical guidance for optimal sampling point
distribution. Numerical experiments are carried out in
Section~\ref{sec:numerics} to show the effectiveness of the proposed
method through three challenging examples: deforming, translating
and rotating or solution-driven interface motions with complex
geometric transformations and high-contrast physical parameters.
Finally, Section~\ref{sec:conclusion} summarizes the key
contributions, discusses advantages and limitations, and outlines
future research directions.

Throughout this paper, we use standard notation for Sobolev spaces
$W^{l,p}(\Psi)$ ($0 \leq l < \infty$, $1 \leq p \leq \infty$) and
their associated norms $\| \cdot \|_{W^{l,p}(\Psi)}$. For $p = 2$,
we use the standard notation $W^{l,2}(\Psi) = H^l(\Psi)$. When $l =
0$, $H^0(\Psi)$ coincides with the standard $L^2$ space, i.e.,
$L^2(\Psi)$. Since we consider time-dependent problems, we also use
the standard notation for the time-dependent Sobolev spaces
$W^{m,q}(0,T;W^{l,p}(\Psi))$ ($0 \leq m < \infty$, $1 \leq q \leq
\infty$) and their associated norms $\| \cdot
\|_{W^{m,q}(W^{l,p}(\Psi))}$. When $p = q = 2$, we use the standard
notation $H^m(0,T; H^l(\Psi))$. Furthermore, if $m = 0$, we use the
notation $L^2(0,T; H^l(\Psi))$.

\section{Model description}\label{sec:model}
This section presents the mathematical model of two-phase flow
problems with moving interfaces, which serves as the foundation of
the proposed PINNs approach. 
Consider $\Omega \subset \mathbb{R}^d$ ($d = 2, 3$) with a convex
polygonal boundary $\partial\Omega$, which is partitioned into two
time-dependent subdomains $\Omega_1(t)$ and $\Omega_2(t)$ separated
by a moving interface $\Gamma(t)$ and satisfies:
\begin{align*}
\Omega = \Omega_1(t) \cup \Omega_2(t), \ \Omega_1(t) \cap
\Omega_2(t) = \emptyset, \ \Gamma(t) = \partial\Omega_1(t) \cap
\partial\Omega_2(t), \ \forall t \in [0,T],
\end{align*}
where $T > 0$ is the terminal time, the external boundaries
$\partial\Omega_i(t) \backslash \Gamma(t)$ remain fixed in time for
both subdomains $\Omega_i(t),\ i = 1, 2$, as illustrated in Figure
\ref{fig:domain}.
\begin{figure}[hbt]
    \begin{center}
        \includegraphics[height=3cm]{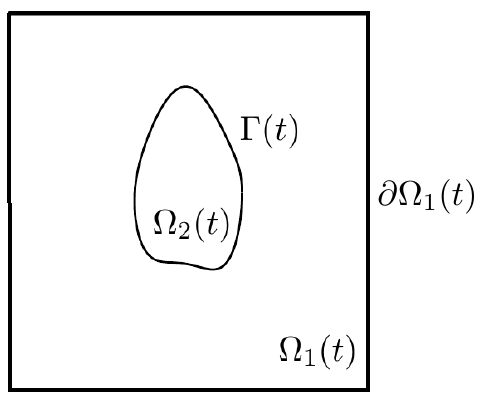}
        \hspace{1cm}
        \includegraphics[height=3cm]{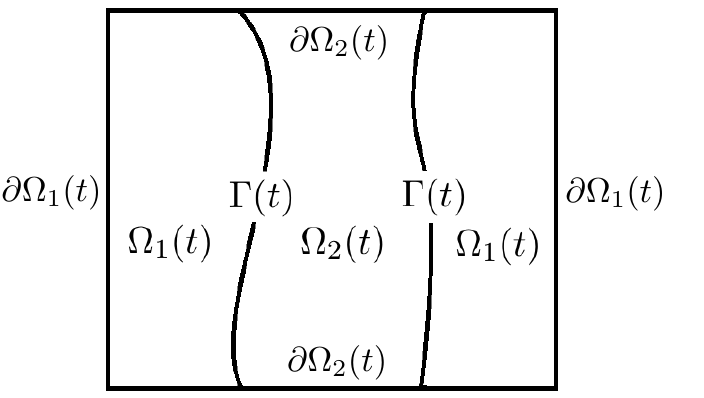}
    \end{center}
    \caption{Two schematic evolving domain decompositions divided by the
time-dependent interface $\Gamma(t)$: the immersed case
(\textbf{left}) and the
        back-to-back case (\textbf{right}).}\label{fig:domain}
\end{figure}

The evolution of the moving interface $\Gamma(t)$ can be classified
into the following two distinct cases based on how the interface
motion is determined.

\noindent\textbf{Case 1: The prescribed interface motion.} In this
case, the moving velocity of the interface, $\bw(\mathbf{x}, t)$, is
explicitly prescribed and known a priori. This includes the
following three common scenarios,
(1) \textit{translation}: $\bw(\mathbf{x}, t) = \mathbf{c}(t)$
    with a spatially independent translational velocity
    $\mathbf{c}(t)$;
(2) \textit{rotation}: $\bw(\mathbf{x}, t) =
    \boldsymbol{\omega}(t) \times (\mathbf{x} - \mathbf{x}_c)$
    with a spatially independent angular velocity $\boldsymbol{\omega}(t)$,
    where $\mathbf{x}_c$ is the center of rotation;
and (3) \textit{deformation}: $\bw(\mathbf{x}, t) =
\mathbf{f}(\mathbf{x}, t)$
    with a given deformation velocity function $\mathbf{f}$.

\noindent\textbf{Case 2: The solution-driven interface motion.} In
this case, the interface motion is determined by the velocity
solution of Navier-Stokes equations in each fluid subdomain through
a coupling mechanism across the moving interface, i.e., the velocity
of moving interface equals to the fluid velocity at the interface,
which must satisfy the kinematic condition:
\begin{equation}
\bw(\mathbf{x}, t) = \bv_1(\mathbf{x}, t) = \bv_2(\mathbf{x}, t),
\quad \mathbf{x} \in \Gamma(t),
\end{equation}
where $\bv_i(\mathbf{x}, t)$ represents the fluid velocity of phase
$i$. In practice, since $\bv_1(\mathbf{x}, t)$ and
$\bv_2(\mathbf{x}, t)$ are unknown variables, $\bw(\mathbf{x}, t)$
is then defined by the numerical solution of $\bv_1(\mathbf{x}, t)$
or $\bv_2(\mathbf{x}, t)$ across $\Gamma(t)$, i.e., $\bw(\mathbf{x},
t)$ is known a posteriori. The current interface position at time
$t$, $\Gamma(t)$, is thus determined by adding the computed
interface displacement, which is the temporal integration of the
velocity field, to the initial interface position $\Gamma(0)$,
defined as $\Gamma(t) = \Gamma(0) + \int_0^t \bw(\mathbf{x}(s), s)\, ds$.

Based on the above geometric configuration and incompressible
Navier-Stokes equations, then a generic two-phase flow moving
interface problem can be defined as follows for the fluid velocity
$\bm{v}_i$ and the fluid pressure $p_i$ $(i=1,2)$: find ${\bm{u}}_i
:=(\bm{v}_i, p_i)\in H^1(0,T;(H^2(\Omega_i(t)))^d)\times
L^\infty(0,T; H^1(\Omega_i(t)))$ such that
\begin{equation}\label{eqn:interface-model}
\left\{
    \begin{array}{rcll}
        \rho_i \left(\frac{\partial\bm{v}_i}{\partial t} + \bm{v}_i \cdot \nabla\bm{v}_i\right) - \nabla \cdot \bm{\sigma}_i &=& \bm{f}_i, &  \text{in} \ \Omega_i(t) \times (0,T], \\
        \nabla \cdot \bm{v}_i &=& 0, &  \text{in} \ \Omega_i(t) \times (0,T], \\
        \bm{v}_1 - \bm{v}_2 &=& \bm{g}_1,& \text{on} \ \Gamma(t) \times [0, T], \\
        \bm{\sigma}_1 \bm{n}_1 + \bm{\sigma}_2 \bm{n}_2 &=&
        \bm{g}_2, & \text{on} \ \Gamma(t) \times [0, T], \\
        \bm{v}_i &=& \bm{v}_i^b,  & \text{on}  \ \partial \Omega_i(t) \backslash \Gamma(t) \times [0,T],\\
        \bm{v}_i(\bm{x}_i, 0) &=& \bm{v}_i^0(\bm{x}_i), & \text{in} \ \Omega_i(0),
    \end{array}
\right.
\end{equation}
where $\bS_i$ denotes the stress tensor of each fluid phase $i$,
given by
\begin{equation}
\bm{\sigma}_i = -p_i \bm{I} + 2\mu_i \bm{D}(\bm{v}_i), \quad
\text{where } \bm{D}(\bm{v}_i) = \frac{1}{2}\left(\nabla\bm{v}_i +
(\nabla\bm{v}_i)^T\right), 
\end{equation}
with $\rho_i$ and $\mu_i$ denoting the density and dynamic viscosity
of fluid phase $i$, respectively, 
$\bm{f}_i$ represents the external body force acting on the fluid
phase $i$, $\bm{n}_1 = -\bm{n}_2$ is the unit outward normal vector
on the interface $\Gamma(t)$ pointing from $\Omega_1(t)$ to
$\Omega_2(t)$ (see Figure~\ref{fig:domain}).
(\ref{eqn:interface-model})$_3$ denotes the kinematic interface
condition that ensures continuity of velocity when $\bm{g}_1 =
\bm{0}$ or allowing for velocity jumps when $\bm{g}_1 \neq \bm{0}$,
and (\ref{eqn:interface-model})$_4$ denotes the dynamic interface
condition that ensures balance of forces when $\bm{g}_2 = \bm{0}$ or
allowing for force jumps when $\bm{g}_2 \neq \bm{0}$. For most
practical applications, we consider no-slip interface conditions
with $\bm{g}_1 = \bm{g}_2 = \bm{0}$. Moreover, $\bm{v}_i^b$ and
$\bm{v}_i^0$ are the Dirichlet boundary value and initial value of
fluid velocity of phase $i$ on the external boundaries and in the
initial domain, respectively. Other types of boundary conditions
(Neumann, Robin, or mixed) can be incorporated as needed for
specific applications.

\section{PINNs/meshfree method} \label{sec:interface-problems}
In this section, we develop a meshfree approach within the frame of
physics-informed neural networks (PINNs) for solving two-phase flow
moving interface problems, which will offer significant advantages
over traditional numerical methods by eliminating the need for
complex mesh generation and providing natural handling of interface
conditions through the neural network architecture. In what follows,
we will introduce the mathematical formulation of PINNs methodology
first for the presented two-phase flow model
(\ref{eqn:interface-model}), then develop its corresponding neural
network approximation to (\ref{eqn:interface-model}).

For the simplicity of notations, we introduce the following
differential operators to represent residuals of governing
equations, interface conditions, boundary conditions and initial
conditions shown in (\ref{eqn:interface-model}), respectively:
\begin{eqnarray}
&&    \mathcal{L}_i(\bm{u}_i) :=
    \begin{pmatrix}
        \rho_i \left(\frac{\partial\bv_i}{\partial
            t}+\bv_i\cdot\nabla\bv_i\right)-\nabla\cdot\bS_i - \bff_i \\
        \nabla \cdot \bv_i
    \end{pmatrix},
\qquad i = 1,2,\\
&&    \mathcal{L}_{\Gamma}(\bm{u}_1, \bm{u}_2) :=
    \begin{pmatrix}
        \bv_1 - \bv_2 - \bg_1 \\
        \bS_1\bn_1+\bS_2\bn_2 - \bg_2
    \end{pmatrix},\\
&&    \mathcal{B}_i(\bm{u}_i) := \bm{v}_i - \bm{v}_i^b, \quad i =
    1,2,\\
&&    \mathcal{I}_i(\bm{u}_i) := \bm{v}_i(\bm{x}_i, 0) -
\bm{v}_i^0(\bm{x}_i), \quad i = 1,2.
\end{eqnarray}
Then, (\ref{eqn:interface-model}) can be reformulated as follows:
\begin{equation}\label{eqn:interface-model1}
\left\{
    \begin{array}{rcll}
        \mathcal{L}_i(\bm{u}_i) &=& \bm{0}, &  \text{in} \ \Omega_i(t) \times (0,T], \ i = 1,2, \\
        \mathcal{L}_{\Gamma}(\bm{u}_1, \bm{u}_2) &=& \bm{0}, & \text{on} \ \Gamma(t) \times [0, T], \\
        \mathcal{B}_i(\bm{u}_i) &=& \bm{0},  & \text{on}  \ \partial \Omega_i(t) \backslash \Gamma(t) \times [0,T], \ i = 1,2, \\
        \mathcal{I}_i(\bm{u}_i) &=& \bm{0}, & \text{in} \ \Omega_i(0),\
        i=1,2.
    \end{array}
\right.
\end{equation}

\subsection{The least-squares formulation} \label{sec:DNN-method}
The core idea of PINNs methodology for the presented two-phase flow
problem with a moving interface is to reformulate the moving
interface problem as a least-squares minimization problem and then
use separate neural networks to approximate the solution of
Navier-Stokes equations in each evolving subdomain that adapts to a
moving interface. This approach extends the framework established
in~\cite{karniadakis2021physics,ZhuHuSun2023} to handle two-phase
flow problems with discontinuous coefficients and complex interface
conditions defined on a moving interface.

First of all, we formulate the following spatiotemporal
least-squares loss functional for (\ref{eqn:interface-model}) or
(\ref{eqn:interface-model1}) with respect to $\tilde{\bm{u}}_i :=
(\tilde{\bm{v}}_i, \tilde{p}_i) \in H^1(0,T;
\left(H^2(\Omega_i(t))\right)^d) \times L^\infty(0,T;
H^1(\Omega_i(t)))$ for $i=1,2$:
\begin{align}
\mathcal{R}(\tilde{\bm{u}}_1, \tilde{\bm{u}}_2) & := \sum_{i=1}^2
\omega_{\mathcal{I}_i} \int_{\Omega_i(0)} \rho_i \left|\bm{v}_i^0 -
\tilde{\bm{v}}_i \right|^2 \mathrm{d} \mathbf{x}_i + \int_{0}^T
\bigg\{ \sum_{i=1}^2 \omega_{\mathcal{L}_i} \int_{\Omega_i(t)}
\bigg[\bigg|\rho_i
\bigg( \frac{\partial \tilde{\bm{v}}_i}{\partial t} +\notag\\
&\qquad\tilde{\bm{v}}_i \cdot \nabla \tilde{\bm{v}}_i \bigg) -
\nabla \cdot {\bS}_i(\tilde{\bu}_i) - \bm{f}_i \bigg|^2  +
\big| \nabla \cdot \tilde{\bm{v}}_i \big|^2\bigg] \mathrm{d} \mathbf{x}_i +\notag\\
& \qquad \sum_{i=1}^2 \omega_{\mathcal{B}_i} \int_{\partial
\Omega_i(t) \backslash \Gamma(t)} \big| \bm{v}_i^b -
\tilde{\bm{v}}_i \big|^2 \mathrm{d} \mathbf{s}_i + \omega_{\Gamma}
\int_{\Gamma(t)}
\bigg(\big| \tilde{\bm{v}}_1 - \tilde{\bm{v}}_2 - \bm{g}_1 \big|^2 +\notag\\
& \qquad  \big| {\bS}_1(\tilde{\bu}_1) \mathbf{n}_1 +
{\bS}_2(\tilde{\bu}_2) \mathbf{n}_2 - \bm{g}_2 \big|^2\bigg)
\mathrm{d} \mathbf{s} \bigg\} \mathrm{d}t \label{LS4twophase}\\
&= \sum_{i=1}^2 \omega_{\mathcal{I}_i} \int_{\Omega_i(0)} \rho_i
\left|\mathcal{I}_i(\tilde{\bm{u}}_i) \right|^2 \mathrm{d}
\mathbf{x}_i + \int_{0}^T \bigg\{ \sum_{i=1}^2
\omega_{\mathcal{L}_i} \int_{\Omega_i(t)}
\big|\mathcal{L}_i(\tilde{\bm{u}}_i)
\big|^2 \mathrm{d} \mathbf{x}_i +\notag\\
& \qquad \sum_{i=1}^2 \omega_{\mathcal{B}_i} \int_{\partial
\Omega_i(t) \backslash \Gamma(t)} \big|
\mathcal{B}_i(\tilde{\bm{u}}_i)\big|^2 \mathrm{d} \mathbf{s}_i +
\omega_{\Gamma} \int_{\Gamma(t)} \big|
\mathcal{L}_{\Gamma}(\tilde{\bm{u}}_1,\tilde{\bm{u}}_2) \big|^2
\mathrm{d} \mathbf{s} \bigg\} \mathrm{d}t, \label{LS4twophase1}
\end{align}
where $\omega_{\mathcal{L}_i}$, $\omega_{\Gamma}$,
$\omega_{\mathcal{B}_i}$ and $\omega_{\mathcal{I}_i}$ are positive
weight coefficients that balance the relative importance of residual
terms of Navier-Stokes equations inside the fluid subdomains
$\Omega_i(t)$, of interface conditions over the interface
$\Gamma(t)$, of boundary conditions over the external boundaries
$\partial \Omega_i(t) \backslash \Gamma(t)$ and of initial
conditions over the initial domains $\Omega_i(0)$, respectively, in
the loss functional. These weight coefficients can be tuned by
common
strategies~\cite{wang2021understanding,wang2022and,mcclenny2020self,ZhuHuSun2023}
including: (1) the normalization by means of the reciprocal of
maximum parameter values in each residual term; (2) the dimensional
consistency across all residual terms; and (3) the adaptive weight
adjustment during training.

Then, the least-squares minimization problem associated with
\eqref{LS4twophase} is formed as follows: find $\bm{u}_i :=
(\bm{v}_i, p_i) \in H^1(0,T; \left(H^2(\Omega_i(t))\right)^d) \times
L^\infty(0,T; H^1(\Omega_i(t)))$, $i=1,2$, such that
\begin{equation}
\mathcal{R}(\bm{u}_1, \bm{u}_2) = \min_{\tilde{\bm{u}}_i \in
H^1(0,T; (H^2(\Omega_i(t)))^d) \times L^\infty(0,T;
H^1(\Omega_i(t))), \, i=1,2} \mathcal{R}(\tilde{\bm{u}}_1,
\tilde{\bm{u}}_2).
\end{equation}

\subsection{Neural Network Architecture}
Our proposed PINNs/meshfree approach employs fully connected neural
networks to approximate primary solutions (velocity and pressure of
each fluid phase) of the presented two-phase flow moving interface
problem in each fluid subdomain, piecewisely, without requiring a
mesh generation. Instead, a $(d+1)$-dimensional spatiotemporal
sampling points/training set, which adapts to the moving interface,
is uniformly or randomly distributed in each evolving fluid
subdomain. Namely, the proposed approach applies separate neural
networks to each fluid subdomain, allowing for different solution
characteristics across the moving interface.

The employed neural network architecture extends the framework that
is established in~\cite{HeLinHu2022,ZhuHuSun2023} for stationary
interface problems to the two-phase flow moving interface problems
by distributing sampling points in two $(d+1)$-dimensional
spatiotemporal fluid subdomains that adapt to the moving interface
in a subtle way. In particular, for the first case of interface
motion whose moving velocity $\bw(\bx,t)$ is explicitly prescribed,
we always know a priori the current position of moving interface
$\Gamma(t)$, thus know the current shape of each fluid subdomain
$\Omega_i(t)$ for $t\in[0,T]$, $i=1,2$, where two sets of sampling
points can be pre-distributed on both sides of the moving interface,
accordingly. As for the second case of interface motion whose moving
velocity $\bw(\bx,t)$ is determined by the fluid velocity
$\bv_i(\bx,t),~i=1,2$, i.e., $\bw=\bv_1=\bv_2$ across the moving
interface $\Gamma(t)$, we need to dynamically track the position of
$\Gamma(t)$ along time by conducting a nonlinear iteration to find
out $\bw|_{ \Gamma(t)}$, as shown later in Algorithm
\ref{alg:solution-dependent-tracking}. Once the moving interface is
tracked, we are then able to establish a $(d+1)$-dimensional
spatiotemporal sampling points set in each evolving fluid subdomain,
accordingly.

By incorporating both spatial and temporal coordinates as inputs to
form the input layer based upon aforementioned spatiotemporal
sampling points/training set in each evolving fluid subdomain, which
is denoted as $\mathbf{X}_i := (\mathbf{x}_i, t) \in
\overline{\Omega_i(t)} \times [0,T] \subset \mathbb{R}^{d+1},\,
i=1,2$, we first construct a fully connected, separated neural
network architecture for each fluid phase $i$ on either side of the
moving interface that consists of different numbers of hidden layers
with linear transformations and nonlinear activation functions.

Introduce the linear transformation $\mathbf{T}_i^l:
\mathbb{R}^{n_i^l} \rightarrow \mathbb{R}^{n_i^{l+1}}$ to the $l$-th
layer, $l = 1, \ldots, L_i$, for each fluid phase $i$ as follows:
\begin{equation}
\mathbf{T}^l_i(\mathbf{X}^l_i) = \mathbf{W}_i^l \mathbf{X}_i^l +
\mathbf{b}_i^l, \quad \text{for } \mathbf{X}_i^l \in
\mathbb{R}^{n_i^l},\,i=1,2,
\end{equation}
where $\mathbf{W}^l_i \in \mathbb{R}^{n_i^{l+1} \times n_i^l}$ are
the weight matrices and $\mathbf{b}^l_i \in \mathbb{R}^{n_i^{l+1}}$
are the bias vectors, $\mathbf{X}_i^{1}=\mathbf{X}_i\in
\mathbb{R}^{n_i^1}=\mathbb{R}^{d+1}$ is the input vector formed by
the spatiotemporal training set of fluid phase $i$. Combined with a
nonlinear activation function $\sigma: \mathbb{R} \rightarrow
\mathbb{R}$, the $l$-th layer of neural networks for the fluid phase
$i$ is then defined as:
\begin{equation}
\mathbf{N}_i^l(\mathbf{X}^l_i) =
\sigma(\mathbf{T}^l_i(\mathbf{X}_i^l)), \quad \text{for } l = 1,
\ldots, L_i,
\end{equation}
where $\sigma$ is applied component-wisely. A complete $L$-layer
deep neural networks is thus constructed as follows for the fluid
phase $i$:
\begin{equation} \label{eqn:NN-structure}
\mathcal{NN}_i(\mathbf{X}_i; \bm{\Theta}_i) = \mathbf{T}_i^L \circ
\mathbf{N}_i^{L-1} \circ \cdots \circ \mathbf{N}_i^2 \circ
\mathbf{N}_i^1(\mathbf{X}_i),\quad i=1,2,
\end{equation}
where 
$\bm{\Theta}_i= \{\mathbf{W}^l_i, \mathbf{b}^l_i : l = 1, \ldots,
L_i\}$ represents all network parameters of the fluid phase $i$, and
$N = \sum_{i=1}^2\sum_{l=1}^{L_i} n_i^{l+1}(n_i^l + 1)$ is the total
number of parameters of a fully connected neural network for the
presented two-phase flow moving interface problem
\eqref{eqn:interface-model}.

\subsection{Discrete formulation based on neural networks}
Now we can employ the above well-defined neural network architecture
to approximate the primary solution
${\tilde{\bm{u}}}_i=(\tilde{\bv}_i,\tilde p_i)$ of each fluid phase
$i$, $i=1,2$, as follows:
\begin{equation} \label{eqn:uNNs}
\tilde{\bm{u}}_1 \approx
\mathcal{U}_{\mathcal{NN}_1}(\mathbf{X}_{1}; \bm{\Theta}_1), \quad
\tilde{\bm{u}}_2 \approx
\mathcal{U}_{\mathcal{NN}_2}(\mathbf{X}_{2}; \bm{\Theta}_2).
\end{equation}
Specifically, we approximate the fluid velocity $\tilde{\bv}_i$ and
the fluid pressure $\tilde p_i$ using separate neural networks for
each fluid phase $i$, define as
\begin{equation} \label{eqn:v-pNNs}
\tilde{\bm{v}}_i \approx
\mathcal{V}_{\mathcal{NN}_i}(\mathbf{X}_{i}; \bm{\Theta}_i), \quad
\tilde{p}_i \approx \mathcal{P}_{\mathcal{NN}_i}(\mathbf{X}_{i};
\bm{\Theta}_i),\quad i=1,2.
\end{equation}
Then, the combined discrete solution vector is defined as:
\begin{equation}
\mathcal{U}_{\mathcal{NN}_i}(\mathbf{X}_{i}; \bm{\Theta}_i) :=
\left(\mathcal{V}_{\mathcal{NN}_i}(\mathbf{X}_{i}; \bm{\Theta}_i),
\mathcal{P}_{\mathcal{NN}_i}(\mathbf{X}_{i}; \bm{\Theta}_i)\right),
\quad i=1,2,
\end{equation}
where we set $n_i^1= d+1$ for the input layer and $n_i^{L+1} = d+1$
for the output layer to accommodate both velocity and pressure
components of of each fluid phase $i$, $i=1,2$.

Since the least-squares loss functional~\eqref{LS4twophase} contains
spatiotemporal integrals that cannot be computed exactly, we
approximate them using Monte Carlo integration with sampling points.
Thus, the following mean squared error (MSE) formulations are
introduced to approximate the corresponding spatiotemporal integrals
of residual terms in~\eqref{LS4twophase}:
\begin{align*}
\mathcal{F}_{\mathcal{I}_i}(\bm{\Theta}_i) & :=
\frac{1}{M_{\mathcal{I}_i}} \sum_{k=1}^{M_{\mathcal{I}_i}} \rho_i
\left| \bm{v}_i^0(\mathbf{X}_{i,k}) -
\mathcal{V}_{\mathcal{NN}_i}(\mathbf{X}_{i,k}; \bm{\Theta}_i)
\right|^2\approx \int_{\Omega_i(0)} \rho_i \left|\mathcal{I}_i(
\tilde{\bm{u}}_i) \right|^2 \mathrm{d} \mathbf{x}_i, \\
\mathcal{F}_{\mathcal{L}_i}(\bm{\Theta}_i) & :=
\frac{1}{M_{\mathcal{L}_i}} \sum_{k=1}^{M_{\mathcal{L}_i}} \bigg[\left|\rho_i \left( \frac{\partial \mathcal{V}_{\mathcal{NN}_i}(\mathbf{X}_{i,k}; \bm{\Theta}_i)}{\partial t} + \mathcal{V}_{\mathcal{NN}_i}(\mathbf{X}_{i,k}; \bm{\Theta}_i) \cdot \nabla \mathcal{V}_{\mathcal{NN}_i}(\mathbf{X}_{i,k}; \bm{\Theta}_i) \right) \right. \\
& \qquad \left. - \nabla \cdot
\bS_i(\mathcal{V}_{\mathcal{NN}_i}(\mathbf{X}_{i,k}; \bm{\Theta}_i),
\mathcal{P}_{\mathcal{NN}_i}(\mathbf{X}_{i,k}; \bm{\Theta}_i)) -
\bm{f}_i(\mathbf{X}_{i,k}) \right|^2
\\& \qquad 
+\left| \nabla \cdot \mathcal{V}_{\mathcal{NN}_i}(\mathbf{X}_{i,k};
\bm{\Theta}_i) \right|^2\bigg] \approx \int_{0}^T \int_{\Omega_i(t)}
\big|\mathcal{L}_i(\tilde{\bm{u}}_i) \big|^2
\mathrm{d}\mathbf{x}_i\,\mathrm{d}t,\\
\mathcal{F}_{\mathcal{B}_i}(\bm{\Theta}_i) & :=
\frac{1}{M_{\mathcal{B}_i}} \sum_{k=1}^{M_{\mathcal{B}_i}} \left|
\bm{v}_i^b(\mathbf{X}_{i,k}) -
\mathcal{V}_{\mathcal{NN}_i}(\mathbf{X}_{i,k}; \bm{\Theta}_i)
\right|^2
\approx \int_{0}^T\int_{\partial \Omega_i(t) \backslash
\Gamma(t)} \big| \mathcal{B}_i(\tilde{\bm{u}}_i) \big|^2 \mathrm{d} \mathbf{s}_i\,\mathrm{d}t, \\
\mathcal{F}_{\Gamma}(\bm{\Theta}_1, \bm{\Theta}_2) & :=
\frac{1}{M_{\Gamma}} \sum_{k=1}^{M_{\Gamma}} \big[\left|
\mathcal{V}_{\mathcal{NN}_1}(\mathbf{X}_{i,k}; \bm{\Theta}_1) -
\mathcal{V}_{\mathcal{NN}_2}(\mathbf{X}_{i,k}; \bm{\Theta}_2) -
\bm{g}_1(\mathbf{X}_{i,k}) \right|^2
\\& \qquad 
+\left| \bS_1(\mathcal{V}_{\mathcal{NN}_1}(\mathbf{X}_{i,k};
\bm{\Theta}_1), \mathcal{P}_{\mathcal{NN}_1}(\mathbf{X}_{i,k};
\bm{\Theta}_1)) \mathbf{n}_1
+\bS_2(\mathcal{V}_{\mathcal{NN}_2}(\mathbf{X}_{i,k};
\bm{\Theta}_2),\right.\\
& \qquad \left. \mathcal{P}_{\mathcal{NN}_2}(\mathbf{X}_{i,k};
\bm{\Theta}_2)) \mathbf{n}_2 - \bm{g}_2(\mathbf{X}_{i,k})
\right|^2\big] \approx\int_0^T\int_{\Gamma(t)} \big|
\mathcal{L}_\Gamma(\tilde{\bm{v}}_1, \tilde{\bm{v}}_2)\big|^2
\mathrm{d} \mathbf{s}\, \mathrm{d}t,
\end{align*}
where $M_{\mathcal{L}_i}$, $M_{\Gamma}$, $M_{\mathcal{B}_i}$ and
$M_{\mathcal{I}_i}$ represent the number of sampling points for the
current interior subdomains, interface, external boundaries and
initial subdomains, respectively. The total number of sampling
points is: $M = \sum\limits_{i=1}^2 (M_{\mathcal{L}_i} +
M_{\mathcal{B}_i} + M_{\mathcal{I}_i}) + M_{\Gamma}$.

Replacing integrals in the least-squares
formulation~\eqref{LS4twophase} with corresponding MSE
approximations as suggested above, we then obtain the following
discrete optimization problem for the two-phase flow moving
interface problem,
\begin{eqnarray}
\mathcal{F}(\bm{\Theta}_1^*,
\bm{\Theta}_2^*)&=&\min\limits_{\bm{\Theta}_1,
\bm{\Theta}_2} \mathcal{F}(\bm{\Theta}_1, \bm{\Theta}_2)\notag\\
&=& \min\limits_{\bm{\Theta}_1, \bm{\Theta}_2} \left[
\sum\limits_{i=1}^2 \left(\omega_{\mathcal{I}_i}
\mathcal{F}_{\mathcal{I}_i}(\bm{\Theta}_i) + \omega_{\mathcal{L}_i}
\mathcal{F}_{\mathcal{L}_i}(\bm{\Theta}_i) + \omega_{\mathcal{B}_i}
\mathcal{F}_{\mathcal{B}_i}(\bm{\Theta}_i)\right)+\right.\notag\\
&&\left.\qquad\qquad\quad\omega_{\Gamma}
\mathcal{F}_{\Gamma}(\bm{\Theta}_1, \bm{\Theta}_2)
\right],\label{eqn:DNN-interface-method}
\end{eqnarray}
which can be solved using gradient-based methods. In this paper, we
employ the Adam optimizer~\cite{kingma2014adam} that is an adaptive
variant of stochastic gradient descent method. The gradients are
computed using automatic differentiation to allow for efficient
computation of derivatives with respect to the network parameters.
Upon convergence to the optimal parameters $\bm{\Theta}_1^*$ and
$\bm{\Theta}_2^*$, we obtain the neural network approximations,
$\mathcal{U}_{\mathcal{NN}_1}^{*}(\mathbf{X}_1;
\bm{\Theta}_1^*)\approx\bm u_1$ and
$\mathcal{U}_{\mathcal{NN}_2}^{*}(\mathbf{X}_2;
\bm{\Theta}_2^*)\approx\bm u_2$ that represent the numerical
solution of the studied two-phase flow moving interface
problem~\eqref{eqn:interface-model}.

\section{Error analysis}\label{sec:error}
In this section, we develop a comprehensive convergence analysis for
the proposed PINNs/meshfree method applied to the presented
two-phase flow moving
interface problems. 
We first introduce the following two assumptions: one is the famous
Universal Approximation Theorem \cite{Cybenko1989,HORNIK1989} that
guarantees neural networks can approximate a very wide class of
functions to any desired degree of accuracy; the other one is to
characterize the accuracy of numerical quadrature scheme that is
adopted for evaluating the spatiotemporal integrals in the loss
functional, numerically.
\begin{assumption} The Universal Approximation Theorem (UAT) \cite{Cybenko1989,HORNIK1989,LESHNO1993}.\label{assumption:UAT}
A family of neural network functions is dense in the space of
continuous functions on a compact domain, i.e., for any continuous
function $f$ defined on a compact subset $\bm K\subset \mathbb{R}^n$
and for any tolerance $\varepsilon>0$, there exists a feedforward
neural network $\hat f$ such that
$$
\|f(x)-\hat f(x)\|_{L^\infty(\bm K)}<\varepsilon.
$$
This approximation is possible with a network that has at least one
hidden layer and an appropriate activation function, and the error
can be made arbitrarily small by increasing the neural network's
capacity.
\end{assumption}

\begin{assumption} \label{assumption:quad}
Consider the following quadrature rule to approximate an integral
$I(g) := \int_{\mathfrak{D}} g(\bm{x}) \ \mathrm{d}\bm{x}$, where
$\mathfrak{D}$ is a spatiotemporal domain, by virtue of $M$
quadrature points $\bm{x}_k \in \mathfrak{D}$ and its corresponding
quadrature weight $\omega_k$ for $1 \leq k \leq M$,
    \begin{equation*}
        I_M(g) : = \sum_{k=1}^M \omega_k g(\bm{x}_k).
    \end{equation*}
Then, the quadrature error between $I(g)$ and $I_M(g)$ is assumed to
be held as
    \begin{equation} \label{eqn:quad}
        \vert  I(g) - I_M(g) \vert \leq C_{quad} \, M^{-\alpha}, \quad \alpha > 0,
    \end{equation}
where the constant $C_{quad}$ depends on the dimension of the domain
and the property of the integrand $g(\bx)$.
\end{assumption}

Now we establish the following convergence theorem for the developed
PINNs approximation to the presented two-phase flow moving interface
problem.
\begin{theorem} \label{thm:error-two-phase-fluid-prescribed}
Let $\bm{v}_i \in
H^1\big(0,T;\left(H^2({\Omega_i(t)})\right)^d\big)$ and $p_i \in
L^\infty(0,T;$ $H^1({\Omega_i(t)}))$ be the classical solution of
the two-phase flow moving interface
problem~\eqref{eqn:interface-model}, and
let~$\mathcal{V}_{\mathcal{NN}_i}^{*}\in H^1\big(0,T;\left((H^2\cap
W^{1,\infty})({\Omega_i(t)})\right)^d\big)$
and~$\mathcal{P}_{\mathcal{NN}_i}^{*}\in L^\infty(0,T;$
$H^1({\Omega_i(t)}))$ be the numerical approximations obtained by
the PINNs/meshfree approach (\ref{eqn:DNN-interface-method}),
$i=1,2$. Then the following error estimation holds:
\begin{align} %
    & \sum_{i=1}^2 \| \bm{v}_i - \mathcal{V}_{\mathcal{NN}_i}^{\ast} \|_{
    {L^2(0,T;L^2(\Omega_i(t)))}}^2  \notag\\
    &\leq C\bigg[\sum_{i=1}^2\bigg(\left( \mathcal{F}_{\mathcal{L}_i}(\bm{\Theta}_i^*) \right)^{\frac{1}{2}}
    + \left( \mathcal{F}_{\mathcal{B}_i}(\bm{\Theta}_i^*) \right)^{\frac{1}{2}}
    + \mathcal{F}_{\mathcal{I}_i}(\bm{\Theta}_i^*)\bigg)+
    \left(\mathcal{F}_{\Gamma}(\bm{\Theta}_1^*,\bm{\Theta}_2^*) \right)^{\frac{1}{2}}+\notag\\
    & \qquad\ \
      \sum_{i=1}^2\bigg( C_{quad}^{\mathcal{L}_i} M_{\mathcal{L}_i}^{-\frac{\alpha_{\mathcal{L}_i}}{2}}
     + C_{quad}^{\mathcal{B}_i} M_{\mathcal{B}_i}^{-\frac{\alpha_{\mathcal{B}_i}}{2}}
     + C_{quad}^{\mathcal{I}_i} M_{\mathcal{I}_i}^{-\alpha_{\mathcal{I}_i}}\bigg)
    + C_{quad}^{\Gamma} M_{\Gamma}^{-\frac{\alpha_{\Gamma}}{2}}
    \bigg].\label{theorem_estimate}
\end{align}
where $C>0$ is a generic constant depending on $\| \bv_i
\|_{H^1(\left(H^2(\Omega_i(t))\right)^d)}$, $\| p_i
\|_{L^\infty(H^1(\Omega_i(t)))}$,
$\rho_i$, $\mu_i$, $\Omega_i(t)\ (i=1,2)$, the Korn constant
$C_{Korn}$, the trace constant $C_{trace}$ and the dimension $d$. If
the interface motion is prescribed with the moving velocity
$\bw(\bx,t)$, then $C$ also depends on
$\|\bw\|_{L^\infty(\Gamma(t))}$.
\end{theorem}
\begin{proof}
Consider the velocity error $\bm{e}_i^{\bm{v}} := \bm{v}_i -
\mathcal{V}_{\mathcal{NN}_i}^{*}$ and the pressure error $e_i^p :=
p_i - \mathcal{P}_{\mathcal{NN}_i}^{*}$, $i=1,2$. Then the following
system of error equations can be derived from
(\ref{eqn:interface-model1}):
\begin{align}
{\rho_i} \left( \frac{\partial \bm{e}_i^{\bm{v}}}{\partial t} +
(\bv_i \cdot \nabla) \be_i^{\bv} +
(\be_i^{\bv} \cdot \nabla) \mathcal{V}^{*}_{\mathcal{NN}_i}\right)&\notag\\
- \nabla \cdot 2 \mu_i \bD(\be_i^{\bv}) + \nabla e_i^p &= \mathcal{R}_{i}^{\bv}, \quad \text{in} \ \Omega_i(t)\times(0,T], \notag \\
\nabla \cdot \bm{e}_i^{\bv} &= \mathcal{R}_{i}^{\nabla \cdot}, \ \, \text{in} \ \Omega_i(t)\times(0,T], \nonumber \\
\be_{1}^{\bv} - \bm{e}_2^{\bv} &= \mathcal{R}_{\Gamma}^{\bv}, \quad \text{on} \ \Gamma(t)\times[0,T], \label{eqn:err-v} \\
\bm{\sigma}_1(\be_1^{\bv}, e_1^p) \bn_1 + \bm{\sigma}_2(\be_2^{\bv}, e_2^p) \bn_2 &= \mathcal{R}_{\Gamma}^{\sigma}, \quad \text{on} \ \Gamma(t)\times[0,T], \notag \\
e_i^{\bm{v}} &= \mathcal{R}_{i}^{\mathcal{B}}, \quad \text{on} \ \partial \Omega_i(t)\backslash \Gamma(t)\times[0,T], \nonumber \\
e_i^{\bm{v}}(\bm{x}_i, 0) & = \mathcal{R}_{i}^{\mathcal{I}}, \quad
\text{in} \ \Omega_i(0), \nonumber
\end{align}
where the residual terms on the right hand side of (\ref{eqn:err-v})
are defined as follows, respectively, $\begin{pmatrix}
    \mathcal{R}_{i}^{\bv} \\
    \mathcal{R}_{i}^{\nabla \cdot}
\end{pmatrix} := -\mathcal{L}_{i}(\mathcal{U}_{\mathcal{NN}_i}^{\ast})$,
$\begin{pmatrix}
    \mathcal{R}_{\Gamma}^{\bv} \\
    \mathcal{R}_{\Gamma}^{\bm{\sigma}}
\end{pmatrix}
:=
-\mathcal{L}_{\Gamma}(\mathcal{U}_{\mathcal{NN}_1}^{\ast},\mathcal{U}_{\mathcal{NN}_2}^{\ast})$,
$\mathcal{R}_{i}^{\mathcal{B}} :=
-\mathcal{B}_i(\mathcal{U}_{\mathcal{NN}_i}^{\ast})$ and
$\mathcal{R}_{i}^{\mathcal{I}} :=
-\mathcal{I}_i(\mathcal{U}_{\mathcal{NN}_i}^{\ast})$, $i = 1,2$.

Multiply the momentum equation (\ref{eqn:err-v})$_1$ by
$\be_{i}^{\bv}$, integrate over the moving domain $\Omega_i(t)$ for
$i=1,2$, conduct integration by parts for the fluid stress term, and
utilize (\ref{eqn:err-v})$_2$, then apply Reynolds transport theorem
to the time derivative of kinetic energy term, appropriately, yield
\begin{align}
&\frac{\mathrm{d}}{\mathrm{d}t} \int_{\Omega_i(t)} \rho_i \frac{|
\be_i^{\bv} |^2}{2} \mathrm{d}\mathbf{x}_i -
\int_{\partial\Omega_i(t)} \rho_i \frac{| \be_i^{\bv} |^2}{2} \bw_i
\cdot \mathbf{n}_i \mathrm{d} s_i
+ \int_{\Omega_i(t)} \rho_i \left( (\bv_i \cdot \nabla) \be_i^{\bv} \right) \cdot \be_i^{\bv} \mathrm{d}\mathbf{x}_i \notag\\
&+ \int_{\Omega_i(t)} \rho_i ( (\be_i^{\bv} \cdot \nabla)
\mathcal{V}_{\mathcal{NN}_i}^{*} ) \cdot \be_i^v
\mathrm{d}\mathbf{x}_i +\int_{\Omega_i(t)} 2 \mu_i |
\bD(\be_i^{\bv}) |^2\mathrm{d}\mathbf{x}_i- \int_{\Omega_i(t)}
\mathcal{R}_{i}^{\nabla \cdot} e_i^p \mathrm{d}\mathbf{x}_i \label{energy-error}\\
&- \int_{\partial \Omega_i(t) \backslash \Gamma(t)}
\mathcal{R}_{i}^{\mathcal{B}} \cdot \bm{\sigma}_i(\be_i^{\bv},
e_i^p) \bn_i \mathrm{d} s_i - \int_{\Gamma(t)} \be_i^{\bv}\cdot
\bm{\sigma}_i(\be_i^{\bv}, e_i^p) \bn_i \mathrm{d} s =
\int_{\Omega_i(t)} \mathcal{R}_{i}^{\bv} \cdot \be_i^{\bv}
\mathrm{d}\mathbf{x}_i,\nonumber
\end{align}
where Reynolds transport theorem is applied to handle the kinetic
energy term $\frac{1}{2}|\be_i^{\bv}|^2$ in a moving domain
$\Omega_i(t)$ with the moving boundary velocity $\bw_i$ on
$\partial\Omega_i(t)$, as shown below:
\begin{equation}\label{Reynolds-Transport}
\frac{\mathrm{d}}{\mathrm{d}t} \int_{\Omega_i(t)}
\frac{|\be_i^{\bv}|^2 }{2} \mathrm{d}\mathbf{x}_i =
\int_{\Omega_i(t)} \frac{\partial\be_i^{\bv}}{\partial
t}\cdot\be_i^{\bv} \mathrm{d}\mathbf{x}_i +
\int_{\partial\Omega_i(t)} \frac{|\be_i^{\bv}|^2}{2} \bw_i \cdot
\mathbf{n}_i \mathrm{d}s_i,
\end{equation}
with $\mathbf{n}_i$ denoting the outward unit normal vector on the
boundary $\partial\Omega_i(t)$.

Sum (\ref{energy-error}) over $i=1,2$, results
\begin{align}
& \frac{\mathrm{d}}{\mathrm{d}t} \left[ \sum_{i=1}^2
\int_{\Omega_i(t)} \rho_i \frac{| \be_i^{\bv} |^2}{2} \mathrm{d}
\bx_i \right] +\sum_{i=1}^2 2 \mu_i \int_{\Omega_i(t)} |
\bD(\be_i^{\bv})
|^2 \mathrm{d} \bx_i= \notag\\
&\quad \sum_{i=1}^2 \int_{\partial\Omega_i(t)} \rho_i \frac{|
\be_i^{\bv} |^2}{2} \bw_i \cdot \mathbf{n}_i \mathrm{d} s_i
-\sum_{i=1}^2 \int_{\Omega_i(t)} \rho_i ((\bv_i \cdot \nabla)
\be_i^{\bv} ) \cdot \be_i^{\bv} \mathrm{d} \bx_i\notag \\
&\quad- \sum_{i=1}^2 \int_{\Omega_i(t)} \rho_i ( (\be_i^{\bv} \cdot
\nabla) \mathcal{V}_{\mathcal{NN}_i}^{*}) \cdot \be_i^{\bv}
\mathrm{d} \bx_i
+ \sum_{i=1}^2 \int_{\Omega_i(t)}
\mathcal{R}_{i}^{\nabla \cdot} e_i^p \mathrm{d}\bx_i\notag \\
&\quad+ \sum_{i=1}^2 \int_{\partial \Omega_i(t) \backslash
\Gamma(t)} \mathcal{R}_{i}^{\mathcal{B}} \cdot
\bm{\sigma}_i(\be_i^{\bv}, e_i^p) \bn_i \mathrm{d} s_i +
\int_{\Gamma(t)} \mathcal{R}_{\Gamma}^{\bm{\sigma}}\cdot \left(
\frac{\be_1^{\bv} + \be_2^{\bv}}{2} \right) \mathrm{d}s \notag\\
& \quad + \int_{\Gamma(t)} \mathcal{R}_{\Gamma}^{\bv}\cdot \left(
\frac{ \bm{\sigma}_1(\be_1^{\bv}, e_1^p) \bn_1 -
\bm{\sigma}_2(\be_2^{\bv}, e_2^p)\bn_2 }{2} \right) \mathrm{d} s +
\sum_{i=1}^2 \int_{\Omega_i(t)} \mathcal{R}_{i}^{\bv} \cdot
\be_i^{\bv} \mathrm{d} \bx_i\notag\\
&=\sum_{k=1}^8 G_k,\label{estimate4twophase}
\end{align}
where the moving boundary velocity $\bw_i\, (i=1,2)$ in the term
$G_1$ depends on the type of interface motion. In what follows, we
estimate terms $G_2$ -- $G_8$ first, then come back to deal with the
term $G_1$ based upon aforementioned two cases of interface motion,
specifically. Apply Cauchy-Schwarz inequality, Young's inequality,
Korn's inequality and Assumption \ref{assumption:UAT}, yields
\begin{align}
   \left|G_2\right| &\leq \sum_{i=1}^2\left(\varepsilon \int_{\Omega_i(t)} \rho_i |(\bv_i \cdot \nabla) \be_i^{\bv}|^2 \mathrm{d} \bx_i
   + 
   C\int_{\Omega_i(t)} \rho_i |\be_i^{\bv}|^2 \mathrm{d} \bx_i\right)  \notag\\
    & \leq \sum_{i=1}^2\left(\varepsilon C_{Korn} \rho_i \| \bv_i \|^2_{L^\infty(\Omega_i(t))}
    \int_{\Omega_i(t)} |\bD(\be_i^{\bv})|^2 \mathrm{d} \bx_i
    + 
    C\int_{\Omega_i(t)} \rho_i |\be_i^{\bv}|^2 \mathrm{d} \bx_i \right),\label{G2}\\
    &\text{where } C_{Korn} \text{ is the Korn constant},\notag\\
    \left|G_3\right| & \leq   \sum_{i=1}^2 \| \nabla \mathcal{V}_{\mathcal{NN}_i}^{*} \|_{L^\infty(\Omega_i(t))}
    \int_{\Omega_i(t)} \rho_i |\be_i^{\bv}|^2 \mathrm{d}\bx_i
    \leq C
    \sum_{i=1}^2\int_{\Omega_i(t)} \rho_i |\be_i^{\bv}|^2 \mathrm{d}
    \bx_i,\label{G3}\\
\left|\sum_{k=4}^8 G_k\right| & \leq C \left[ \sum_{i=1}^2
\left(\int_{\Omega_i(t)} |\mathcal{R}_{i}^{\bv}|^2 \mathrm{d} \bx_i
+ \int_{\Omega_i(t)} |\mathcal{R}_{i}^{\nabla \cdot}|^2 \mathrm{d}
\bx_i + \int_{\partial \Omega_i(t) \backslash \Gamma(t)}
|\mathcal{R}_{i}^{\mathcal{B}}|^2 \mathrm{d}s_i\right) \right.\notag\\
&\left.\qquad+ \int_{\Gamma(t)} |\mathcal{R}_{\Gamma}^{\bv}|^2
\mathrm{d} s + \int_{\Gamma(t)}
|\mathcal{R}_{\Gamma}^{\bm{\sigma}}|^2 \mathrm{d} s
\right]^{\frac{1}{2}},
\end{align}
where $C$ denotes a generally positive constant depending on $\|
\bv_i \|_{H^1(\left(H^2(\Omega_i(t))\right)^d)}$, $\| p_i
\|_{L^\infty(H^1(\Omega_i(t)))}$,
$\rho_i$, $\mu_i$, $\Omega_i(t)\ (i=1,2)$ and the dimension $d$.

Finally, we estimate the term $G_1$ for the following two cases of
interface motion, respectively.

\textbf{Case 1: The prescribed interface motion.} In this case, the
moving boundary velocity, $\bw_1=\bw_2=\bw$, is known a priori on
$\Gamma(t)$, and $\bw_i=0$ on $\partial\Omega_i(t) \backslash
\Gamma(t)$, $i=1,2$, i.e., the external boundary of each fluid
subdomain is fixed. Therefore, we can estimate $G_1$ as follows,
\begin{align}
G_1 &= \int_{\Gamma(t)} \frac{\rho_1| \be_1^{\bv} |^2-\rho_2|
\be_2^{\bv} |^2}{2} \bw \cdot \mathbf{n}_1 \, \mathrm{d} s\notag\\
&= \int_{\Gamma(t)}
\left(\frac{\rho_1}{2}(\be_1^{\bv}+\be_2^{\bv})\cdot
(\be_1^{\bv}-\be_2^{\bv})+\frac{\rho_1-\rho_2}{2}|
\be_2^{\bv} |^2\right)\bw \cdot \mathbf{n}_1 \, \mathrm{d} s\notag\\
&= \int_{\Gamma(t)} \frac{\rho_1}{2}(\be_1^{\bv}+\be_2^{\bv})\cdot
\mathcal{R}_{\Gamma}^{\bv}\,\bw \cdot \mathbf{n}_1 \, \mathrm{d}
s+\int_{\Gamma(t)}\frac{\rho_1-\rho_2}{2}|
\be_2^{\bv} |^2\bw \cdot \mathbf{n}_1 \, \mathrm{d} s\notag\\
&\leq 
C\|\bw\|_{L^\infty(\Gamma(t))}\left(\sum_{i=1}^2
\|\be_i^{\bv}\|_{L^2(\Gamma(t))}\|\mathcal{R}_{\Gamma}^{\bv}\|_{L^2(\Gamma(t))}
+\|\be_2^{\bv} \|_{L^2(\Gamma(t))}^2\right)\notag\\
&\leq 
C\|\bw\|_{L^\infty(\Gamma(t))}C_{trace}\left(\sum_{i=1}^2
\|\be_i^{\bv}\|_{H^1(\Omega_i(t))}\|\mathcal{R}_{\Gamma}^{\bv}\|_{L^2(\Gamma(t))}\right.\notag\\
&\left.\quad
+
\left({\varepsilon_{\Gamma}}\|\be_2^{\bv}\|_{H^1(\Omega_2(t))}
+\frac{1}{{\varepsilon_{\Gamma}}}
\|\be_2^{\bv}\|_{L^2(\Omega_2(t))}\right)^2\right)\notag\\
&\leq 
C\|\bw\|_{L^\infty(\Gamma(t))}C_{trace}\left(\sum_{i=1}^2
\|\be_i^{\bv}\|_{H^1(\Omega_i(t))}\|\mathcal{R}_{\Gamma}^{\bv}\|_{L^2(\Gamma(t))}\right.\notag\\
&\left.\quad +\varepsilon^2_{\Gamma}
C^2_{Korn}\|\bD(\be_2^{\bv})\|^2_{L^2(\Omega_2(t))}
+\frac{1}{{\varepsilon}^2_{\Gamma}}\|\be_2^{\bv}\|^2_{L^2(\Omega_2(t))}\right),\notag\\
&\leq
C\left(\sum_{i=1}^2\|\mathcal{R}_{\Gamma}^{\bv}\|_{L^2(\Gamma(t))}
+\|\be_2^{\bv}\|^2_{L^2(\Omega_2(t))}\right)
+\varepsilon\|\bD(\be_2^{\bv})\|^2_{L^2(\Omega_2(t))},\label{G1}
\end{align}
where the trace theorem and the refined trace estimate \cite[Page
27]{Xu1989} are applied, $C_{trace}$ is the constant of trace
estimate, and the previous introduced generic constant $C>0$ depends
on
$\|\bw\|_{L^\infty(\Gamma(t))}$, 
$C_{trace}$ and $C_{Korn}$ as well. 

\textbf{Case 2: The solution-driven interface motion.} In this case,
we have $\bw_i = \mathcal{V}_{\mathcal{NN}_i}^{*}\,(i=1,2)$ on the
moving interface $\Gamma(t)$. Then, by applying the divergence
theorem to the term $G_1$, we obtain
\begin{align}
G_1&=\sum_{i=1}^2 \int_{\partial\Omega_i(t)} \rho_i \frac{|
\be_i^{\bv} |^2}{2}
\mathcal{V}_{\mathcal{NN}_i}^{*} \cdot \mathbf{n}_i \mathrm{d} s = \sum_{i=1}^2 \int_{\Omega_i(t)} \nabla \cdot \left( \rho_i \frac{| \be_i^{\bv} |^2}{2} \mathcal{V}_{\mathcal{NN}_i}^{*} \right) \mathrm{d} \bx \notag \\
&= \sum_{i=1}^2 \int_{\Omega_i(t)} \rho_i \mathcal{V}_{\mathcal{NN}_i}^{*} \cdot \nabla \left( \frac{| \be_i^{\bv} |^2}{2} \right) \mathrm{d} \bx + \sum_{i=1}^2 \int_{\Omega_i(t)} \rho_i \frac{| \be_i^{\bv} |^2}{2} \nabla \cdot \mathcal{V}_{\mathcal{NN}_i}^{*} \mathrm{d} \bx \notag \\
&= \sum_{i=1}^2 \int_{\Omega_i(t)} \rho_i
\left((\mathcal{V}_{\mathcal{NN}_i}^{*} \cdot \nabla)
\be_i^{\bv}\right) \cdot \be_i^{\bv} \mathrm{d} \bx + \sum_{i=1}^2
\int_{\Omega_i(t)} \rho_i \frac{| \be_i^{\bv} |^2}{2} \nabla \cdot
\mathcal{V}_{\mathcal{NN}_i}^{*} \mathrm{d} \bx\notag\\
&=H_1+H_2,\label{G1-1}
\end{align}
where similarly with the estimation of $G_2$, we have
\begin{align}
   H_1 &\leq \sum_{i=1}^2\left(\varepsilon \int_{\Omega_i(t)}
   \rho_i |(\mathcal{V}_{\mathcal{NN}_i}^{*} \cdot \nabla) \be_i^{\bv}|^2 \mathrm{d} \bx_i
   + 
   C\int_{\Omega_i(t)} \rho_i |\be_i^{\bv}|^2 \mathrm{d} \bx_i\right)  \notag\\
    & \leq \sum_{i=1}^2\left(\varepsilon C_{Korn} \rho_i \| \mathcal{V}_{\mathcal{NN}_i}^{*} \|^2_{L^\infty(\Omega_i(t))}
    \int_{\Omega_i(t)} |\bD(\be_i^{\bv})|^2 \mathrm{d} \bx_i
    + 
    C\int_{\Omega_i(t)} \rho_i |\be_i^{\bv}|^2 \mathrm{d} \bx_i
    \right),\label{H1}
\end{align}
and
\begin{align}
    H_2 & \leq \sum_{i=1}^2 \| \nabla\cdot \mathcal{V}_{\mathcal{NN}_i}^{*} \|_{L^\infty(\Omega_i(t))}
    \int_{\Omega_i(t)} \frac{\rho_i}{2} |\be_i^{\bv}|^2 \mathrm{d}\bx_i
    \leq C
    \sum_{i=1}^2\int_{\Omega_i(t)} \rho_i |\be_i^{\bv}|^2 \mathrm{d}
    \bx_i.\label{H2}
\end{align}

Then, combine (\ref{G2})-(\ref{H2}), and choose sufficiently small
$\varepsilon$, 
(\ref{estimate4twophase}) 
becomes:
\begin{align}
& \frac{\mathrm{d}}{\mathrm{d}t} \left[
\sum_{i=1}^2 \int_{\Omega_i(t)} 
|\be_i^{\bv} |^2
\mathrm{d} \bx \right]+\sum_{i=1}^2 \int_{\Omega_i(t)} |
\bD(\be_i^{\bv}) |^2 \mathrm{d} \bx_i \notag \\
    & \leq C \left[ \left(\sum_{i=1}^2 \left(\int_{\Omega_i(t)}
    |\mathcal{R}_{i}^{\bv}|^2 \mathrm{d} \bx  +
    \int_{\Omega_i(t)} |\mathcal{R}_{i}^{\nabla \cdot}|^2 \mathrm{d} \bx +
    \int_{\partial \Omega_i(t) \backslash \Gamma(t)}
    |\mathcal{R}_{i}^{\mathcal{B}}|^2 \mathrm{d}s \right) \right.\right.\notag\\
    & \left.\left.\qquad\quad+ \int_{\Gamma(t)}\left(|\mathcal{R}_{\Gamma}^{\bv}|^2
    +|\mathcal{R}_{\Gamma}^{\bm{\sigma}}|^2\right) \mathrm{d} s
    \right)^{\frac{1}{2}} + \sum_{i=1}^2 \int_{\Omega_i(t)} |
    \be_i^{\bv}|^2 \mathrm{d} \bx
    \right],\label{Error-BeforeGronwall}
\end{align}

Integrate (\ref{Error-BeforeGronwall}) over time for any $0\leq t
\leq \tau$, where $0<\tau\leq T$, yields
\begin{align*}
    & \sum_{i=1}^2 \int_{\Omega_i(\tau)} | \be_i^{\bv}(\bx_i,\tau) |^2 \mathrm{d} \bx
+\int_0^{\tau}\sum_{i=1}^2 \int_{\Omega_i(t)} |
\bD(\be_i^{\bv}(\bx_i,t)) |^2
\mathrm{d} \bx_i\mathrm{d} t\\
& \leq \sum_{i=1}^2 \int_{\Omega_i(0)} |
\mathcal{R}_{i}^{\mathcal{I}} |^2 \mathrm{d} \bx  \\
    & \quad+ C \int_0^\tau\left[ \left(\sum_{i=1}^2 \left(\int_{\Omega_i(t)}
    |\mathcal{R}_{i}^{\bv}|^2 \mathrm{d} \bx  +
    \int_{\Omega_i(t)} |\mathcal{R}_{i}^{\nabla \cdot}|^2 \mathrm{d} \bx +
    \int_{\partial \Omega_i(t) \backslash \Gamma(t)}
    |\mathcal{R}_{i}^{\mathcal{B}}|^2 \mathrm{d}s \right) \right.\right.\notag\\
    & \left.\left.\qquad\quad+ \int_{\Gamma(t)}\left(|\mathcal{R}_{\Gamma}^{\bv}|^2
    +|\mathcal{R}_{\Gamma}^{\bm{\sigma}}|^2\right) \mathrm{d} s
    \right)^{\frac{1}{2}} + \sum_{i=1}^2 \int_{\Omega_i(t)} |
    \be_i^{\bv}|^2 \mathrm{d} \bx
    \right]\mathrm{d} t
\end{align*}

Applying the Gr\"{o}nwall's inequality, we obtain
\begin{align*}
    & \sum_{i=1}^2 \int_{\Omega_i(\tau)} | \be_i^{\bv}(\bx_i,\tau) |^2 \mathrm{d} \bx
+\int_0^{\tau}\sum_{i=1}^2 \int_{\Omega_i(t)} |
\bD(\be_i^{\bv}(\bx_i,t)) |^2
\mathrm{d} \bx_i\mathrm{d} t\\
& \leq \sum_{i=1}^2 \int_{\Omega_i(0)} |
\mathcal{R}_{i}^{\mathcal{I}} |^2 \mathrm{d} \bx  \\
    & \quad+ C \int_0^T \left(\sum_{i=1}^2 \left(\int_{\Omega_i(t)}
    |\mathcal{R}_{i}^{\bv}|^2 \mathrm{d} \bx  +
    \int_{\Omega_i(t)} |\mathcal{R}_{i}^{\nabla \cdot}|^2 \mathrm{d} \bx +
    \int_{\partial \Omega_i(t) \backslash \Gamma(t)}
    |\mathcal{R}_{i}^{\mathcal{B}}|^2 \mathrm{d}s \right) \right.\notag\\
    & \left.\qquad\quad+ \int_{\Gamma(t)}\left(|\mathcal{R}_{\Gamma}^{\bv}|^2
    +|\mathcal{R}_{\Gamma}^{\bm{\sigma}}|^2\right) \mathrm{d} s
    \right)^{\frac{1}{2}}\mathrm{d} t
\end{align*}
Integrate again in time over $0\leq\tau\leq T$, leads to
\begin{eqnarray} \label{ine:two-phase-flow-pde-assumption}
&&    
\sum_{i=1}^2 \| \be_i^{\bv}\|^2_{L^2(L^2(\Omega_i(t)))}
 \leq C\bigg[\sum_{i=1}^2 \int_{\Omega_i(0)} |
\mathcal{R}_{i}^{\mathcal{I}} |^2 \mathrm{d} \bx  +\notag\\
    && \quad \int_0^T \left(\sum_{i=1}^2 \left(\int_{\Omega_i(t)}
    |\mathcal{R}_{i}^{\bv}|^2 \mathrm{d} \bx  +
    \int_{\Omega_i(t)} |\mathcal{R}_{i}^{\nabla \cdot}|^2 \mathrm{d} \bx +
    \int_{\partial \Omega_i(t) \backslash \Gamma(t)}
    |\mathcal{R}_{i}^{\mathcal{B}}|^2 \mathrm{d}s \right)+ \right.\notag\\
    && \left.\quad \int_{\Gamma(t)}\left(|\mathcal{R}_{\Gamma}^{\bv}|^2
    +|\mathcal{R}_{\Gamma}^{\bm{\sigma}}|^2\right) \mathrm{d} s
    \right)^{\frac{1}{2}}\mathrm{d} t\bigg]\label{doubleint_time0}\\
&& \leq C\bigg[\sum_{i=1}^2\left(\|\mathcal{R}_{i}^{\mathcal{I}}
\|^2_{L^2(\Omega_i(0))}+
 \|\mathcal{R}_{i}^{\bv}\|_{L^2(L^2(\Omega_i(t)))} +
 \|\mathcal{R}_{i}^{\nabla
\cdot}\|_{L^2(L^2(\Omega_i(t)))}+\right. \notag\\
&& \left. \quad \|\mathcal{R}_{i}^{\mathcal{B}}\|_{L^2(L^2(\partial
\Omega_i(t) \backslash \Gamma(t)))} \right)+
\|\mathcal{R}_{\Gamma}^{\bv}\|_{L^2(L^2(\Gamma(t)))}
    +\|\mathcal{R}_{\Gamma}^{\bm{\sigma}}\|_{L^2(L^2(\Gamma(t)))}\bigg].\label{doubleint_time}
\end{eqnarray}

To estimate the terms on the right-hand side of
(\ref{doubleint_time0}), we employ Assumption~\ref{assumption:quad},
resulting in
\begin{align*}
\int_0^T \left(\int_{\Omega_i(t)}\left(
    |\mathcal{R}_{i}^{\bv}|^2+|\mathcal{R}_{i}^{\nabla \cdot}|^2\right)
    \mathrm{d} \bx\right)^{\frac{1}{2}}\mathrm{d}t&=
\|\mathcal{L}_i(\mathcal{U}_{\mathcal{NN}_i}^{\ast})
\|_{L^2(0,T;L^2(\Omega_i(t)))} \\
&= \left( \int_{\Omega_i(t) \times [0,T]}
|\mathcal{L}_i(\mathcal{U}_{\mathcal{NN}_i}^{\ast}(\bm{X}_i;
\bm{\Theta}_i^*)|^2 \, \mathrm{d} \bm{X}_i
\right)^{\frac{1}{2}} \\
&\leq \left( \mathcal{F}_{\mathcal{L}_i}(\bm{\Theta}_i^*)
\right)^{\frac{1}{2}} + C_{quad}^{\mathcal{L}_i}
M_{\mathcal{L}_i}^{-\frac{\alpha_{\mathcal{L}_i}}{2}}.
\end{align*}
Apply similar estimates to the remaining terms on the right-hand
side of (\ref{doubleint_time0}), yields
\begin{align*} %
    & \sum_{i=1}^2 \| \bm{v}_i - \mathcal{V}_{\mathcal{NN}_i}^{\ast} \|_{
    {L^2(0,T;L^2(\Omega_i(t)))}}^2  \notag\\
    &\leq C\bigg[\sum_{i=1}^2\bigg(\left( \mathcal{F}_{\mathcal{L}_i}(\bm{\Theta}_i^*) \right)^{\frac{1}{2}}
    + \left( \mathcal{F}_{\mathcal{B}_i}(\bm{\Theta}_i^*) \right)^{\frac{1}{2}}
    + \mathcal{F}_{\mathcal{I}_i}(\bm{\Theta}_i^*)\bigg)+
    \left(\mathcal{F}_{\Gamma}(\bm{\Theta}_1^*,\bm{\Theta}_2^*) \right)^{\frac{1}{2}}+\notag\\
    & \qquad\ \
      \sum_{i=1}^2\bigg( C_{quad}^{\mathcal{L}_i} M_{\mathcal{L}_i}^{-\frac{\alpha_{\mathcal{L}_i}}{2}}
     + C_{quad}^{\mathcal{B}_i} M_{\mathcal{B}_i}^{-\frac{\alpha_{\mathcal{B}_i}}{2}}
     + C_{quad}^{\mathcal{I}_i} M_{\mathcal{I}_i}^{-\alpha_{\mathcal{I}_i}}\bigg)
    + C_{quad}^{\Gamma} M_{\Gamma}^{-\frac{\alpha_{\Gamma}}{2}}
    \bigg].
\end{align*}
The desired error estimation (\ref{theorem_estimate}) is then
obtained.
\end{proof}

Theorem \ref{thm:error-two-phase-fluid-prescribed} provides both a
priori and a posteriori error bounds which depend on the relative
magnitude between the optimization error (i.e., the first four terms
on the right hand side of (\ref{theorem_estimate}) and the
quadrature error (i.e., the last four terms on the right hand side
of (\ref{theorem_estimate}). When the optimization problem is solved
accurately, the optimization error terms becomes negligible in
comparison with the quadrature error terms, producing a priori
convergence rates with respect to the number of sampling points. In
this scenario, the a priori numerical error is also called the
generalization error (gen-error). On the other hand, if the
optimization error terms dominate, then the error estimation serves
as a posteriori error indicator for adaptive refinement strategies
based upon the numerical solution's behavior.

\section{Numerical experiments} \label{sec:numerics}
To demonstrate the effectiveness and capacity of the proposed
PINNs/meshfree method and validate the theoretical convergence
results obtained in Section \ref{sec:error}, in this section we
conduct numerical experiments through three examples of two-phase
flow interface problem with various moving interface configurations
and physical parameters. The developed PINNs approach for all
numerical examples are implemented using
PyTorch~\cite{NEURIPS2019_9015} and are executed on a computer
equipped with Intel 5th generation CPU and Tesla A40 GPU. The fully
connected DNN structures used in our numerical experiments, unless
otherwise specified, consist of $3$ hidden layers with $50$ neurons
in each hidden layer, and employ the $\tanh$ activation function on
each layer.

In addition, the training process of the developed PINNs approach
for two-phase flow moving interface problems is conducted in two
phases: a pretraining phase followed by a main training phase.
During the pretraining phase, all weight coefficients
$\omega_{\mathcal{L}_i}$, $\omega_{\Gamma}$,
$\omega_{\mathcal{B}_i}$, and $\omega_{\mathcal{I}_i}$ are set to
$1$, with a fixed learning rate of $0.001$ for the first $20,000$
epochs. In the subsequent main training phase within the next
$80,000$ epochs, the weight coefficients are adjusted such that
$\omega_{\Gamma} = \omega_{\mathcal{B}_i} = 10$, except Example 3 in
which an adaptive strategy is employed to configure all weight
coefficients as well as the learning rate in the main training
phase. In addition, the ADAM optimizer is employed throughout both
training phases, entirely.

\subsection{Example 1: The case of prescribed moving interface without jump coefficients}\label{num:example1}
We first consider a two-phase flow interface problem with an moving
interface undergoing the following prescribed deforming, translating
and rotating motion in the $(2+1)$-th dimensional spatiotemporal
space $\mathbb{R}^2\times[0,T]$:
\begin{equation}\label{elliptic-move}
\begin{array}{rcl}
x(\theta,t) &=& x_0+v_xt + a(t)\cos\theta\cos(\omega t)-b(t)\sin\theta\sin(\omega t), \\
y(\theta,t) &=& y_0+v_yt + a(t)\cos\theta\sin(\omega t)+b(t)\sin\theta\cos(\omega t), \\
z(\theta,t) &=& t, \quad \text{for } 0\leq\theta\leq2\pi,\ 0\leq
t\leq T,
\end{array}
\end{equation}
which actually forms a three-dimensional, inclined, generalized
helicoid with a deforming and translating elliptical profile curve
whose semi-major axis, $a(t)$, and semi-minor axis, $b(t)$, vary in
time, where $(x_0,y_0)$ is the initial ellipse center,
$\textbf{v}=(v_x,v_y)^\top$ is the translational velocity of the
elliptical profile curve in $xy$-plane, $\omega$ is the rotational
angular velocity of the helicoid, and, (\ref{elliptic-move})$_3$
indicates that the translational velocity of the helicoid along
$z$-axis (i.e., the time axis) is just $1$.
\begin{figure}[hbt]
\centering
\includegraphics[width=5cm, height=4cm]{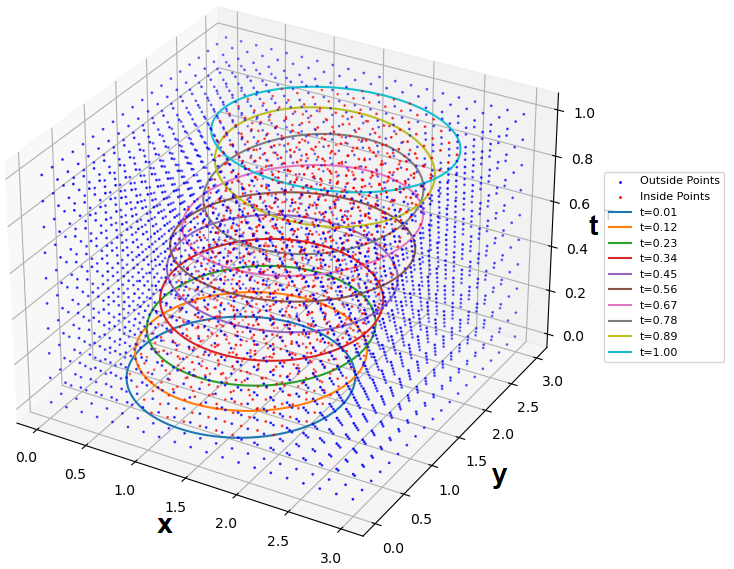}
\caption{The training set of Example 1 separated by a prescribed
generalized helicoid with a deforming and translating elliptical
profile curve. The elliptical interfaces that are projections of the
helicoid on $xy$-plane at different times are plotted with different
colors, while their interiors are uniformly covered by red dashed
line segments at different angles $2\pi t$ along the time axis.}
 \label{fig:elliptic-points}
\end{figure}

In this example, we choose $x_0=y_0=1.2,\ v_x=v_y=0.6,\
\omega=2\pi,\ a(t)=1+0.1t,\ b(t)=1-0.1t$, $\Omega := [0,3] \times
[0,3]\subset \mathbb{R}^2$ and set $T=1$ as the terminal time. Then
at each time $t\in[0,1]$, by projecting the above described helicoid
onto the $xy$-plane $z=t$, we have $\Omega_2(t) := \big\{(x,y)\in
\Omega \ \big|$ $\left(\frac{x-1.2-0.6t}{1 + 0.1t}\right)^2$ $+
\left(\frac{y-1.2-0.6t}{1 - 0.1t}\right)^2 < 1 \big\}$,
$\Omega_1(t)=\Omega\backslash\overline{\Omega_2(t)}$, and
$\Gamma(t):=\big\{(x,y) \in
\partial\Omega_1(t)\cap\partial\Omega_2(t) \ \big|$
$\left(\frac{x-1.2-0.6t}{1+0.1t}\right)^2+\left(\frac{y-1.2-0.6t}{1
- 0.1t}\right)^2 = 1 \big\}$. 
However, from a holistic three-dimensional spatiotemporal
perspective, the computational domain of this example are the
exterior and interior subdomains separated by the helicoidal
interface defined in (\ref{elliptic-move}). Therefore, as
illustrated in Figure \ref{fig:elliptic-points}, the training set to
be employed to implement the developed PINNs/meshfree method for
this example is separately sampled inside two $(2+1)$-th dimensional
spatiotemporal subdomains, on their boundaries, initial domains and
the helicoidal interface between them that is a generalized helicoid
with a deforming and translating elliptical profile curve, based
upon the definition of interface motion (\ref{elliptic-move}).

Next, we set the physical parameters $\rho_1=\rho_2 = 1$ and $\mu_1
= \mu_2 = 1$, and appropriately choose $\bm{f}_i,\ \bg_i,\ \bv_i^b,\
\bv_i^0$, $i=1,2$, in (\ref{eqn:interface-model}) to make sure the
following functions $(\bm{v}_i,\ p_i)$, $i=1,2$, are the exact
solution to this example:
\begin{align}
&\bm{v}_1 =
\begin{pmatrix}
e^t \sin x  \cos y \\
-e^t \cos x  \sin y
\end{pmatrix}, \ \quad\quad
p_1 =   e^t  \sin x  \sin y, \notag\\
&\bm{v}_2 =
\begin{pmatrix}
    \cos t  \cos x  \cos y \\
    \cos t  \sin x  \sin y
\end{pmatrix}, \quad
p_2 =   \cos t  \cos(x+y),\label{realsolution1}
\end{align}
which induces $\bg_i\neq 0,\ i=1,2$, leading to the case of
jump-type interface conditions across $\Gamma(t)$. We address that
uniform physical parameters are chosen in this example to
investigate the fundamental performance of the developed
PINNs/meshfree method for a two-phase flow moving interface problem
without jump coefficients.

In addition, to investigate the convergence behavior of the
developed PINNs approach in terms of the generalization error, we
increasingly choose $M_{\mathcal{L}_i}$, $M_{\mathcal{B}_i}$,
$M_{\Gamma}$ and $M_{\mathcal{I}_i}$ as number of sampling points
inside two $(2+1)$-th dimensional spatiotemporal subdomains, on
their boundaries, the helicoidal interface and initial domains,
respectively, as displayed in Table~\ref{tab:NSNS-elliptic-1vs1}.
Two types of numerical errors are included in this table: the
relative generalization error in $L^2$ norm and the loss error in
$l^2$ norm. We can see from the last two columns of
Table~\ref{tab:NSNS-elliptic-1vs1} that errors of the loss function
are much smaller than the generalization errors, meaning that we are
in a regime where the optimization problem is solved fairly
accurately while the quadrature error dominates the numerical
approximation of PINNs approach. In this case, our error estimation
derived in Section~\ref{sec:error} mainly serves as a prior error
estimation. Furthermore, we intend to show that the quadrature
errors coming from loss terms of interface conditions, boundary
conditions and initial conditions play an important role as well,
besides the quadrature errors arising from loss terms of PDEs in the
interior spatiotemporal domains, noting that dimensions of the
interface, the boundary, and the initial subdomains are lower than
that of the interior spatiotemporal subdomains. In our experiments,
as shown in Table~\ref{tab:NSNS-elliptic-1vs1}, when we keep the
number of interior sampling point $M_{\mathcal{L}_i}$ as a constant
while increasing the number of sampling points on the interface,
$M_{\Gamma}$, on the boundaries, $M_{\mathcal{B}_i}$, and in the
initial subdomains, $M_{\mathcal{I}_i}$, $i=1,2$, we can observe
that the overall generalization errors decrease as expected,
although only a very small number of them stagnate or even increase,
which may indicate that the quadrature errors from the interior
domain terms take over the dominance when the number of sampling
points on the interface/boundaries and in the initial subdomains are
more than enough. On the other hand,
Table~\ref{tab:NSNS-elliptic-1vs1} exhibits that, with the same
setup of $M_{\mathcal{B}_i}$, $M_{\Gamma}$ and $M_{\mathcal{I}_i}$,
the larger $M_{\mathcal{L}_i}$ always produces the smaller
generalization error, which tells us that the quadrature errors
generated from the interior spatiotemporal subdomains play the
dominant role, in general. Overall, those observations are
consistent with our theoretical result derived in Section
\ref{sec:error}.
\begin{table}[H]
\centering \caption{Loss error and generalization error (gen-error)
vs \# of sampling points in Example 1}
\label{tab:NSNS-elliptic-1vs1}
\begin{tabular}{ c c c c c c}
\hline \hline $M_{\mathcal{L}_i}$ & $M_{\mathcal{B}_i}$ &
$M_{\Gamma}$&  $M_{\mathcal{I}_i}$  &  Gen-Error & Loss Error
\\ \hline
$10 \times 10 \times 5$ & $4 \times 4 \times 5$ & $4 \times 5$  & $4 \times 4$  & 7.32e-02  & 3.85e-05 \\
$10 \times 10 \times 5$ & $8 \times 4 \times 5$ & $8 \times 5$  & $8 \times 8$  & 6.51e-02  & 3.41e-05 \\
$10 \times 10 \times 5$ & $16 \times 4 \times 5$ & $16 \times 5$  & $16 \times 16$  & 5.28e-02  & 2.93e-05 \\
$10 \times 10 \times 5$ & $32 \times 4 \times 5$ & $32 \times 5$  &
$32 \times 32$  & 4.92e-02  & 3.12e-05 \\ \hline
$20 \times 20 \times 10$ & $4 \times 4 \times 10$ & $4 \times 10$  & $4 \times 4$  & 5.45e-02  & 2.76e-05 \\
$20 \times 20 \times 10$ & $8 \times 4 \times 10$ & $8 \times 10$  & $8 \times 8$  & 4.10e-02  & 2.32e-05 \\
$20 \times 20 \times 10$ & $16 \times 4 \times 10$ & $16 \times 10$  & $16 \times 16$  & 4.25e-02  & 2.17e-05 \\
$20 \times 20 \times 10$ & $32 \times 4 \times 10$ & $32 \times 10$
& $32 \times 32$  & 3.48e-02  & 1.86e-05 \\ \hline
$40 \times 40 \times 20$ & $4 \times 4 \times 20$ & $4 \times 20$  & $4 \times 4$  & 3.57e-02  & 1.92e-05 \\
$40 \times 40 \times 20$ & $8 \times 4 \times 20$ & $8 \times 20$  & $8 \times 8$  & 2.83e-02  & 1.65e-05 \\
$40 \times 40 \times 20$ & $16 \times 4 \times 20$ & $16 \times 20$  & $16 \times 16$  & 2.41e-02  & 1.71e-05 \\
$40 \times 40 \times 20$ & $32 \times 4 \times 20$ & $32 \times 20$  & $32 \times 32$  & 1.98e-02  & 1.43e-05 \\
\hline \hline
\end{tabular}
\end{table}




Figure \ref{fig:elliptic-pinn-10} illustrates numerical results of
fluid velocity and pressure predicted by the developed PINNs
approach on the $xy$-plane $z=1$ (i.e., at the terminal time) of the
outer and inner subdomains, respectively. In the meanwhile, error
distributions between the PINNs output and the exact solution of
fluid velocity and pressure are shown in Figure
\ref{fig:elliptic-error-10} at the same plane as well, which tells
that the largest errors mainly exist near the moving elliptical
interface, explaining why we set the weight coefficient of the loss
function for interface conditions, $\omega_{\Gamma}$, is higher than
others.
\begin{figure}[H]
    \centering
\includegraphics[width=12cm, height=6cm]{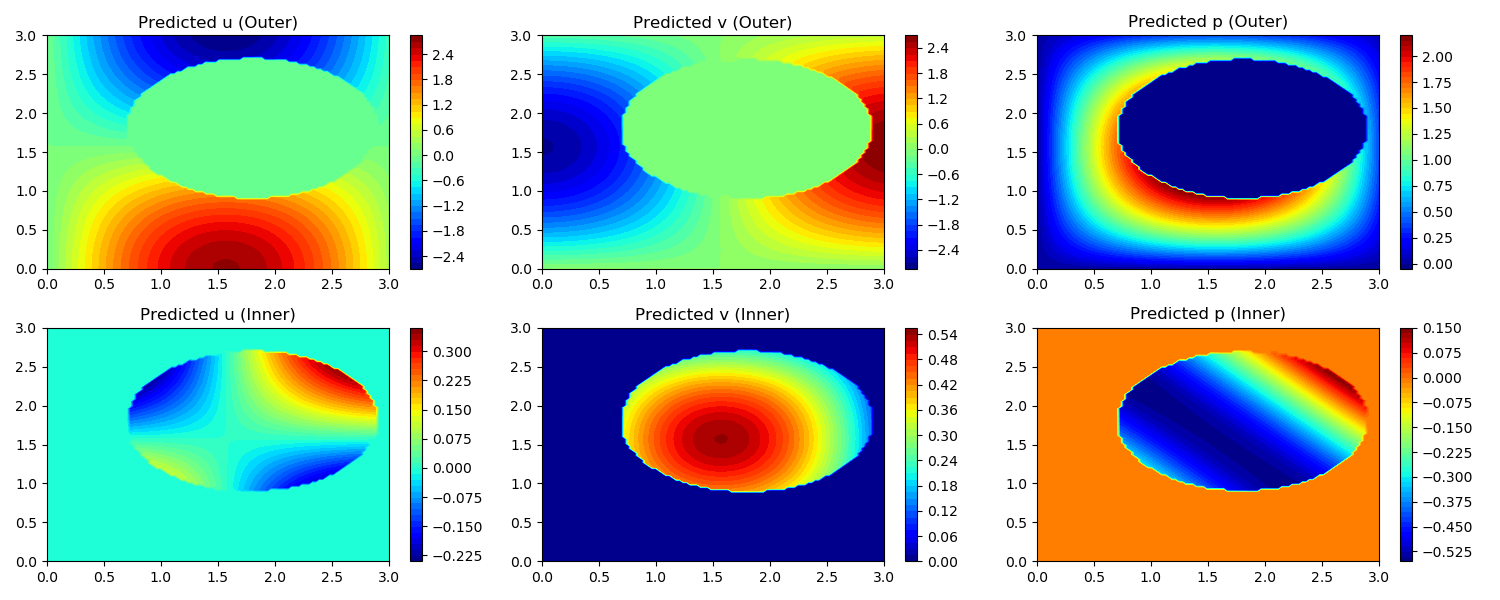}
\caption{PINNs outputs of Example 1 on the terminal planes of the
outer and inner subdomains, respectively. \textbf{Left}: The
horizontal velocity $u$; \textbf{Middle}: the vertical velocity $v$;
\textbf{Right}: the pressure $p$; \textbf{Upper}: outside the
elliptical interface; \textbf{Lower}: inside the elliptical
interface.} \label{fig:elliptic-pinn-10}
\end{figure}
\begin{figure}[H]
    \centering
\includegraphics[width=12cm, height=6cm]{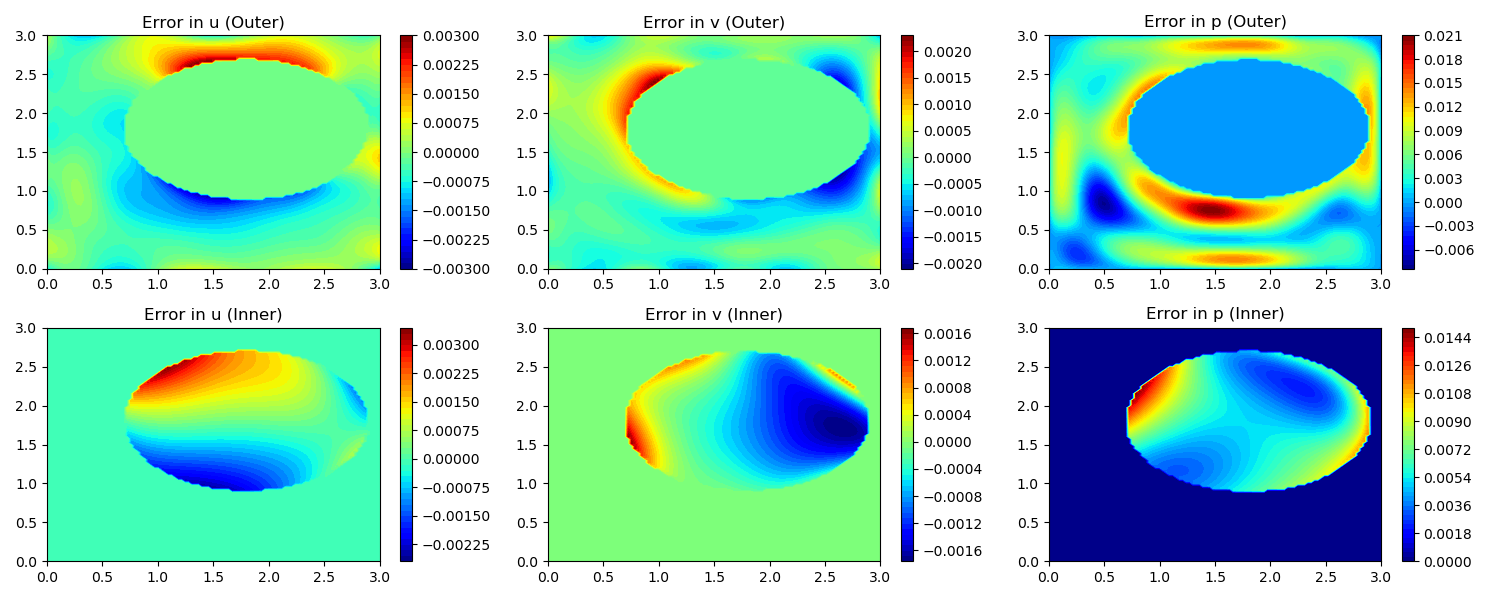}
\caption{Error distributions between PINNs output and exact solution
of Example 1 on the terminal planes of the outer and inner
subdomains, respectively. \textbf{Left}: The horizontal velocity
$u$; \textbf{Middle}: the vertical velocity $v$; \textbf{Right}: the
pressure $p$; \textbf{Upper}: outside the elliptical interface;
\textbf{Lower}: inside the elliptical interface.}
\label{fig:elliptic-error-10}
\end{figure}

\subsection{Example 2: The case of prescribed moving interface with jump coefficients}\label{num:example2}
Now we consider another example of the two-phase flow interface
problem with a prescribed deforming, translating and rotating
interface motion in time, in the meanwhile, the density and
viscosity of two-phase flow are no longer identical but jumping
across the moving interface. Precisely, the interface motion is
prescribed below in the $(2+1)$-th dimensional spatiotemporal space
$\mathbb{R}^2\times[0,T]$:
\begin{equation}\label{5star-move}
\begin{array}{rcl}
x(\theta,t) &=& x_0+v_xt + r(\theta,t)\cos\theta\cos(\omega t)-r(\theta,t)\sin\theta\sin(\omega t), \\
y(\theta,t) &=& y_0+v_yt + r(\theta,t)\cos\theta\sin(\omega t)+r(\theta,t)\sin\theta\cos(\omega t), \\
z(\theta,t) &=& t,\text{ where }r(\theta,t) = 1 -
0.3t\cos(5\theta),\, \theta\in[0,2\pi],\, t\in[0, T].
\end{array}
\end{equation}
Essentially, (\ref{5star-move}) forms an inclined generalized
helicoid with a deforming and translating five-fold-modulation
cyclic-harmonic profile curve, which is traced out in the $(2+1)$-th
dimensional spatiotemporal space. The bottom of such a helicoid
(i.e., at $t=0$) is a circle with the initial center $(x_0,y_0)$ and
the radius $1$. As time evolves, the definition of polar function
$r(\theta,t)$ in (\ref{5star-move}) produces a cyclic-harmonic curve
with a five-fold modulation in the $xy$-plane $z=t$ that deforms and
translates along time with the translational velocity
$\textbf{v}=(v_x,v_y)^\top$, while rotating about the z-axis (i.e.,
the time axis) with the angular velocity $\omega$ as well as
translating along the z-axis with the velocity $1$. Eventually, the
aforementioned helicoid is formed. In this example, we still choose
$x_0=y_0=1.2,\ v_x=v_y=0.6,\ \omega=2\pi,\ T=1$, and pick $\Omega :=
[0,3.5] \times [0,3.5]$ this time.

As shown in Figure \ref{fig:5star-points}, the training/sampling
points are separately sampled inside two $(2+1)$-th dimensional
spatiotemporal subdomains, on their boundaries, initial domains and
the helicoidal interface between them that is a generalized helicoid
with a deforming and translating five-fold-modulation
cyclic-harmonic profile curve along time, based upon the definition
of interface motion (\ref{5star-move}).
\begin{figure}[hbt]
\centering
\includegraphics[width=5cm, height=4cm]{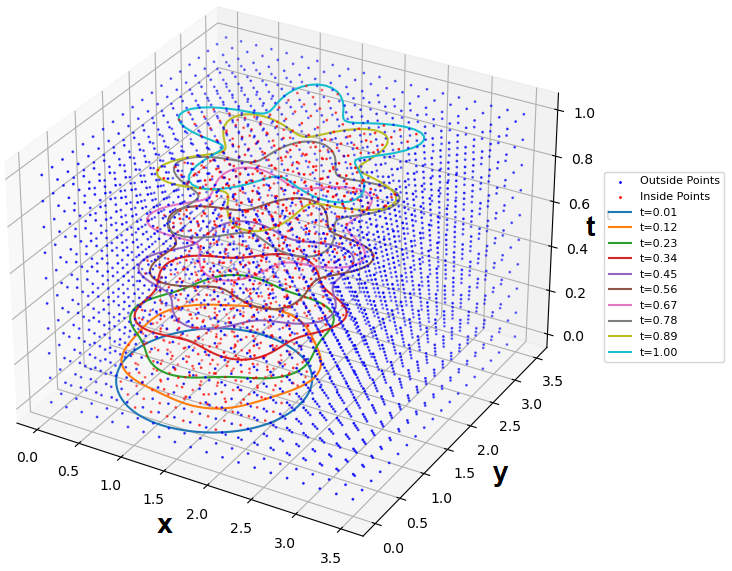}
\caption{The training set of Example 2 separated by a prescribed
generalized helicoid with a deforming and translating
five-fold-modulation cyclic-harmonic profile curve. The
cyclic-harmonic curves with a five-fold modulation, which are
projections of the helicoid onto $xy$-planes at different times, are
plotted with different colors, while their interiors are uniformly
covered by red dashed line segments at different angles $2\pi t$
along the time axis.}
 \label{fig:5star-points}
\end{figure}

By appropriately choosing $\bm{f}_i,\ \bg_i,\ \bv_i^b,\ \bv_i^0$,
$i=1,2$, in (\ref{eqn:interface-model}), we select the same exact
solution as Example 1 for this example, however with high-contrast
jumping coefficients $\rho_1=1$, $\rho_2 = 1000$, $\mu_1 = 1$ and
$\mu_2 = 1000$. Then, we apply the developed PINNs/meshfree method
to this example in terms of the same incremental sampling number
setting and training strategy as in Example 1. The obtain numerical
errors are illustrated in
Table~\ref{tab:NSNS-5star-high-contrast-no-obs} and
Figure~
\ref{fig:high-contrast-5star-no-obs-error-10}, where we find out
that numerical errors deteriorate significantly this time, loss
errors and generalization errors are as high as the order of
$10^{-2}$ and $10^{-1}$, respectively, in comparison with $10^{-5}$
and $10^{-2}$ for the case of no-jump-coefficient in Example 1. In
addition, the largest errors are all concentrated near the
interface, especially the fluid pressure, with generalization errors
on the order of $10^{-1}$. This significantly suppresses the error
order of fluid velocity, $10^{-2}$, therefore dominating total
generalization errors around the moving interface. This numerical
challenge is precisely because the physical coefficients change so
drastically at the moving interface that the interface condition
(\ref{eqn:interface-model})$_4$ is poorly trained, resulting in a
poor approximation of the fluid pressure.\vspace{-.5cm}
\begin{table}[H]
\centering \caption{Loss error and generalization error vs \# of
sampling points in Example 2 without the aid of observation points}
\label{tab:NSNS-5star-high-contrast-no-obs}
\begin{tabular}{ c c c c c c}
\hline \hline $M_{\mathcal{L}_i}$ & $M_{\mathcal{B}_i}$ &
$M_{\Gamma}$&  $M_{\mathcal{I}_i}$  &  Gen-Error & Loss Error
\\ \hline
{$10 \times 10 \times 5$} & {$4 \times 4 \times 5$} & {$4 \times 5$}  & {$4 \times 4$}  & {6.82e-01}  & {1.08e-02} \\
{$10 \times 10 \times 5$} & {$8 \times 4 \times 5$} & {$8 \times 5$}  & {$8 \times 8$}  & {5.31e-01}  & {9.42e-03} \\
{$10 \times 10 \times 5$} & {$16 \times 4 \times 5$} & {$16 \times 5$}  & {$16 \times 16$}  & {7.15e-01}  & {1.12e-02} \\
{$10 \times 10 \times 5$} & {$32 \times 4 \times 5$} & {$32 \times
5$}  & {$32 \times 32$} & {4.67e-01} & {8.89e-03} \\
\hline
{$20 \times 20 \times 10$} & {$4 \times 4 \times 10$} & {$4 \times 10$}  & {$4 \times 4$}  & {5.23e-01}  & {9.27e-03} \\
{$20 \times 20 \times 10$} & {$8 \times 4 \times 10$} & {$8 \times 10$}  & {$8 \times 8$}  & {5.56e-01}  & {1.02e-02} \\
{$20 \times 20 \times 10$} & {$16 \times 4 \times 10$} & {$16 \times 10$}  & {$16 \times 16$}  & {5.01e-01}  & {9.03e-03} \\
{$20 \times 20 \times 10$} & {$32 \times 4 \times 10$} & {$32 \times
10$}  & {$32 \times
32$}  & {4.78e-01}  & {8.76e-03} \\
\hline
{$40 \times 40 \times 20$} & {$4 \times 4 \times 20$} & {$4 \times 20$}  & {$4 \times 4$}  & {3.56e-01}  & {7.89e-03} \\
{$40 \times 40 \times 20$} & {$8 \times 4 \times 20$} & {$8 \times 20$}  & {$8 \times 8$}  & {2.91e-01}  & {7.43e-03} \\
{$40 \times 40 \times 20$} & {$16 \times 4 \times 20$} & {$16 \times 20$}  & {$16 \times 16$}  & {2.45e-01}  & {7.16e-03} \\
{$40 \times 40 \times 20$} & {$32 \times 4 \times 20$} & {$32 \times 20$}  & {$32 \times 32$}  & {2.67e-01}  & {7.29e-03} \\
\hline \hline
\end{tabular}
\end{table}
\begin{figure}[H]
    \centering
\includegraphics[width=12cm, height=6cm]{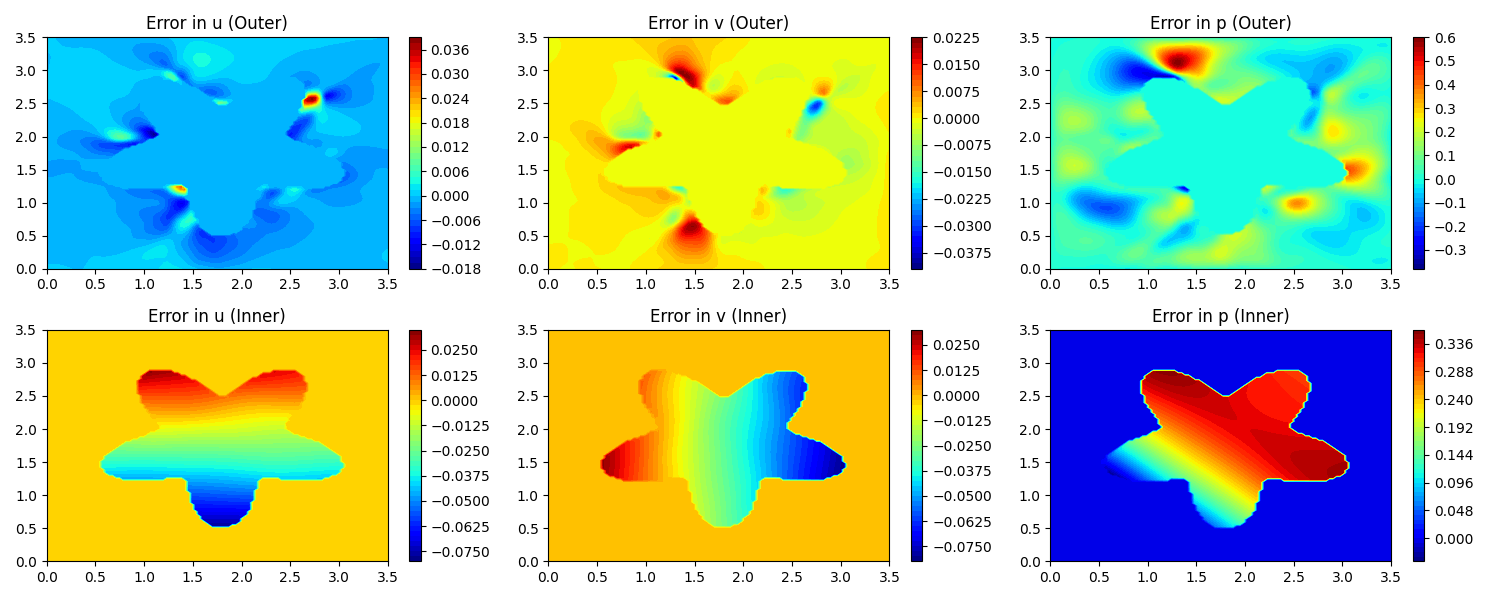}
\caption{Without the aid of observation points, error distributions
between PINNs output and exact solution of Example 2 on the terminal
planes of the outer and inner subdomains, respectively.
\textbf{Left}: The horizontal velocity $u$; \textbf{Middle}: the
vertical velocity $v$; \textbf{Right}: the pressure $p$;
\textbf{Upper}: outside the five-fold-modulation cyclic-harmonic
interface; \textbf{Lower}: inside the five-fold-modulation
cyclic-harmonic interface.}
\label{fig:high-contrast-5star-no-obs-error-10}
\end{figure}

Similar numerical phenomena have been elucidated in
\cite{ZhuHuSun2023}, where one solution method to overcome the
problem of insufficient numerical precision induced by jumping
coefficients is to add some observation data points for the fluid
pressure near the interface to improve the PINNs prediction's
accuracy on the pressure across the interface, and consequently on
the velocity as well.
Therefore, to improve the numerical accuracy for this example, we
introduce a certain number of observation points on the moving
interface by assigning the exact solution of fluid pressure as the
known data to each observation point. In particular, we pick 5
observation points at the five-fold-modulation cyclic-harmonic
profile curve which is the projection of the helicoid onto
$xy$-plane $z=t$, with angles $\theta \in \{0, 2\pi/5, 4\pi/5,
6\pi/5, 8\pi/5\}$ and values of time $t\in\{0.01, 0.12, 0.23, 0.34,
0.45, 0.56, 0.67, 0.78, 0.89, 1.00\}$. Therefore, 50 observation
points are taken with the coordinates $(x,y) = (1.2+0.6t + (1 -
0.3t\cos(5\theta))\cos\theta$, $1.2+0.6t + (1 -
0.3t\cos(5\theta))\sin\theta)$, which all locate on the moving
interface $\Gamma(t)$, exactly.

Then, we introduce the following sub-loss function based on these
observation data points,
$$\mathcal{F}_{\mathcal{D}}(\bm{\Theta}_2) :=
\frac{1}{M_{\mathcal{D}}} \sum_{k=1}^{M_{\mathcal{D}}} \left|
p_2(\mathbf{X}_{2,k}) -
\mathcal{P}_{\mathcal{NN}_2}(\mathbf{X}_{2,k}; \bm{\Theta}_2)
\right|^2,$$ and add it to the total loss functional,
$\mathcal{F}(\bm{\Theta}_1, \bm{\Theta}_2)$, with the corresponding
weight coefficient $\omega_{\mathcal{D}}$, defined as
\begin{align}
&\mathcal{F}(\bm{\Theta}_1, \bm{\Theta}_2)=\sum\limits_{i=1}^2
\left(\omega_{\mathcal{I}_i}
\mathcal{F}_{\mathcal{I}_i}(\bm{\Theta}_i) + \omega_{\mathcal{L}_i}
\mathcal{F}_{\mathcal{L}_i}(\bm{\Theta}_i) + \omega_{\mathcal{B}_i}
\mathcal{F}_{\mathcal{B}_i}(\bm{\Theta}_i)\right)+\omega_{\Gamma}
\mathcal{F}_{\Gamma}(\bm{\Theta}_1, \bm{\Theta}_2)\notag\\
&\qquad\qquad\qquad+\omega_{\mathcal{D}}
\mathcal{F}_{\mathcal{D}}(\bm{\Theta}_2),\label{newtotalloss}
\end{align}
where $M_{\mathcal{D}}=50,\ \omega_{\mathcal{D}}=1$ in this example.

By virtue of the above observation data points and the new loss
functional (\ref{newtotalloss}), we reapply the developed
PINNs/meshfree method to this example and attain new numerical
results as shown in Table~\ref{tab:NSNS-5star-high-contrast-obs} and
Figures \ref{fig:high-contrast-5star-pinn-10},
\ref{fig:high-contrast-5star-error-10}, where we can see that both
the loss errors and the generalization errors drop substantially
down to the order of $10^{-5}$ and $10^{-2}$ again, respectively. In
specific, we observe that as the number of sampling points increase,
numerical errors exhibit the same decreasing trend as shown in Table
\ref{tab:NSNS-elliptic-1vs1}, i.e., with the same setup of
$M_{\mathcal{L}_i}$, the larger $M_{\mathcal{B}_i}$, $M_{\Gamma}$
and $M_{\mathcal{I}_i}$ delivers the smaller generalization error.
On the other hand, with the same setup of $M_{\mathcal{B}_i}$,
$M_{\Gamma}$ and $M_{\mathcal{I}_i}$, the larger $M_{\mathcal{L}_i}$
produces the smaller generalization error. These numerical phenomena
reveal again that quadrature errors coming from loss terms of
interface conditions, boundary conditions and initial conditions
play a significant role as well, in addition to the dominant
quadrature errors arising from the interior spatiotemporal
subdomains. The obtained theoretical result in Section
\ref{sec:error} is thus validated again.

\begin{table}[H]
\centering \caption{Loss error and generalization error vs \# of
sampling points in Example 2 aided by observation points}
\label{tab:NSNS-5star-high-contrast-obs}
\begin{tabular}{ c c c c c c}
\hline \hline $M_{\mathcal{L}_i}$ & $M_{\mathcal{B}_i}$ &
$M_{\Gamma}$&  $M_{\mathcal{I}_i}$  &  Gen-Error & Loss Error
\\ \hline
$10 \times 10 \times 5$ & $4 \times 4 \times 5$ & $4 \times 5$  & $4 \times 4$  & 6.62e-02  & 3.95e-05 \\
$10 \times 10 \times 5$ & $8 \times 4 \times 5$ & $8 \times 5$  & $8 \times 8$  & 5.93e-02  & 3.76e-05 \\
$10 \times 10 \times 5$ & $16 \times 4 \times 5$ & $16 \times 5$  & $16 \times 16$  & 5.34e-02  & 3.92e-05 \\
$10 \times 10 \times 5$ & $32 \times 4 \times 5$ & $32 \times 5$  & $32 \times 32$  & 4.76e-02  & 3.58e-05 \\ \hline
$20 \times 20 \times 10$ & $4 \times 4 \times 10$ & $4 \times 10$  & $4 \times 4$  & 4.94e-02  & 3.24e-05 \\
$20 \times 20 \times 10$ & $8 \times 4 \times 10$ & $8 \times 10$  & $8 \times 8$  & 3.82e-02  & 2.97e-05 \\
$20 \times 20 \times 10$ & $16 \times 4 \times 10$ & $16 \times 10$  & $16 \times 16$  & 4.13e-02  & 2.85e-05 \\
$20 \times 20 \times 10$ & $32 \times 4 \times 10$ & $32 \times 10$  & $32 \times 32$  & 3.28e-02  & 2.76e-05 \\ \hline
$40 \times 40 \times 20$ & $4 \times 4 \times 20$ & $4 \times 20$  & $4 \times 4$  & 3.52e-02  & 2.43e-05 \\
$40 \times 40 \times 20$ & $8 \times 4 \times 20$ & $8 \times 20$  & $8 \times 8$  & 2.47e-02  & 2.28e-05 \\
$40 \times 40 \times 20$ & $16 \times 4 \times 20$ & $16 \times 20$  & $16 \times 16$  & 2.12e-02  & 2.35e-05 \\
$40 \times 40 \times 20$ & $32 \times 4 \times 20$ & $32 \times 20$  & $32 \times 32$  & 1.58e-02  & 2.05e-05 \\
\hline \hline
\end{tabular}
\end{table}
\begin{figure}[H]
    \centering
\includegraphics[width=12cm, height=6cm]{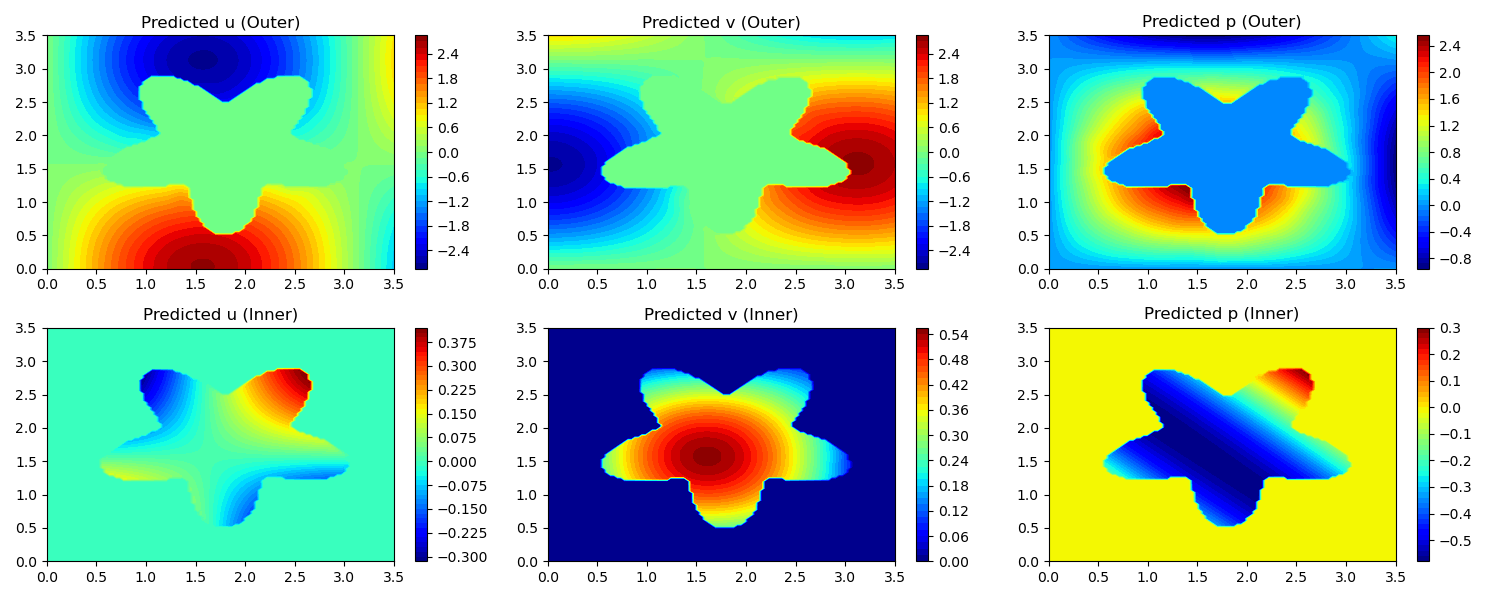}
\caption{The observation-point-aided PINNs outputs of Example 2 on
the terminal planes of the outer and inner subdomains, respectively.
\textbf{Left}: The horizontal velocity $u$; \textbf{Middle}: the
vertical velocity $v$; \textbf{Right}: the pressure $p$;
\textbf{Upper}: outside the five-fold-modulation cyclic-harmonic
interface; \textbf{Lower}: inside the five-fold-modulation
cyclic-harmonic interface.} \label{fig:high-contrast-5star-pinn-10}
\end{figure}


\begin{figure}[H]
    \centering
\includegraphics[width=12cm, height=6cm]{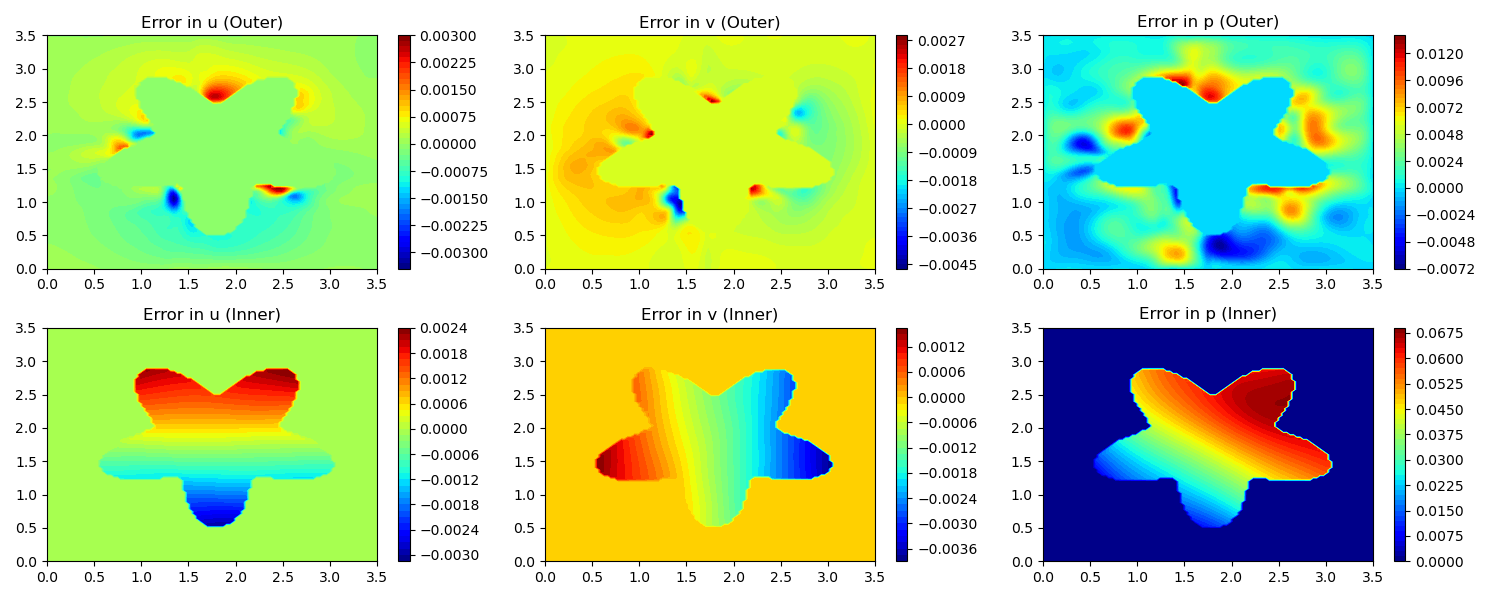}
\caption{Error distributions between the observation-point-aided
PINNs output and exact solution of Example 2 on the terminal planes
of the outer and inner subdomains, respectively. \textbf{Left}: The
horizontal velocity $u$; \textbf{Middle}: the vertical velocity $v$;
\textbf{Right}: the pressure $p$; \textbf{Upper}: outside the
five-fold-modulation cyclic-harmonic interface; \textbf{Lower}:
inside the five-fold-modulation cyclic-harmonic interface.}
\label{fig:high-contrast-5star-error-10}
\end{figure}

Further, in the outer and inner subdomains that are separated by the
five-fold-modulation cyclic-harmonic interface on the terminal plane
$t=1$, Figure \ref{fig:high-contrast-5star-pinn-10} illustrates
numerical results of velocity and pressure predicted by the
developed PINNs approach for the case of jumping coefficients with
the aid of observation points, whose errors from the exact solution
of this example are shown in Figure
\ref{fig:high-contrast-5star-error-10}. We can see that although the
largest errors are still mainly concentrated near the moving
interface, the individual error orders of velocity and pressure are
reduced to the same level as in Example 1, $10^{-3}$ and $10^{-2}$,
respectively.

In summary, Example 2 shows a more challenging case of the
prescribed interface motion due to the complex geometry of moving
interface with a concave feature and high contrast coefficients, in
comparison with Example 1 in which the interface motion exhibits a
relatively simple and convex elliptical shape without jump
coefficients. Therefore, compared to the case of
no-jump-coefficient, numerical errors of this example are increased
in the first place as expected when dealing with large jumps in
physical coefficients across the moving interface. However, with the
help of some observation data points on the moving interface for the
fluid pressure only, the developed PINNs approach still achieves a
reasonable accuracy that is comparable with the case of
no-jump-coefficient, which means that adding appropriate observation
data points can significantly improve the accuracy of PINNs
approach, especially in the case of high contrast coefficients with
a complex geometry. So far, both Examples 1 and 2 demonstrate the
effectiveness of the proposed PINNs/meshfree method for two-phase
flow moving interface problems with various configurations of
interface motion and physical parameters, including the prescribed
deforming, translating and rotating interface motion with or without
high-contrast physical coefficients.

\subsection{Example 3: The case of solution-driven moving interface}\label{num:example6}
In this section, we consider a more realistic also the most
challenging example of two-phase flow interface problem in which the
interface motion is driven by the solution. Precisely, the current
position of interface is entirely determined by the fluid
displacement across the interface that is the time integration of
fluid velocity field, $\int_0^t \bm v(\bx|_{\Gamma(\tau)},\tau) \,
d\tau$, with an initial fluid displacement of zero, calculated by
$\Gamma(t) = \Gamma(0) + \int_0^t \bm v(\bx|_{\Gamma(\tau)},\tau) \,
d\tau$ for $t\in(0,T]$. In practice, such an example can represent
the most physically realistic scenario of two-phase flow moving
interface problems, where the interface evolution between two kinds
of fluid flow is solely driven by the solution of fluid velocity
across the interface.

In this example, we define the initial interface position $\Gamma(0)
$ as a circle with the center $(1.5, 1.5)$ and the radius $1.0$
within $\Omega := [0,3] \times [0,3]\subset \mathbb{R}^2$, and set
$T=1$. Then, $\Omega_2(0) := \{(x,y)\in \Omega \ |\ (x-1.5)^2+
(y-1.5)^2 < 1 \}$,
$\Omega_1(0)=\Omega\backslash\overline{\Omega_2(0)}$, and
$\Gamma(0):=\{(x,y)\in
\partial\Omega_1(t)\cap\partial\Omega_2(t) \ |\ (x-1.5)^2+ (y-1.5)^2 = 1 \}$.
We further set the physical parameters as $\rho_1=1$, $\rho_2 =
1000$, $\mu_1 = 1$, and $\mu_2 = 1000$ to represent a high contrast
scenario, and take the following functions $(\bm{v}_i,\ p_i)$,
$i=1,2$, as the exact solution for velocity and pressure fields of
the two-phase flow interface problem by appropriately choosing
$\bm{f}_i,\ \bg_i,\ \bv_i^b,\ \bv_i^0$, $i=1,2$, in
(\ref{eqn:interface-model}),
\begin{align}
\bm{v}_1 = \bm{v}_2=
\begin{pmatrix}
\cos t \sin x \cos y \\
-\cos t \cos x \sin y
\end{pmatrix}, \ p_1 =   e^t \sin x \sin y, \ p_2 =   \cos t \cos(x +
y),\label{realsolution3}
\end{align}
which induces $\bg_1=0$ and $\bg_2\neq 0$.


To describe an interface motion that is entirely driven by the
solution of two-phase flow along the time, we need to update the
moving interface $\Gamma(t)$, which is evolved from $\Gamma(0)$, by
virtue of the PINNs prediction of fluid velocity as follows,
\begin{equation}\label{interface-update}
\begin{array}{rcl}
\Gamma(t)|_{xy\text{-plane}:\, z=t} &=& \Gamma(0) + \int_0^t
\mathcal{V}_{\mathcal{NN}_2}(\mathbf{X}_{2}|_{\Gamma(\tau)};
\bm{\Theta}_2) \, d\tau\\
&\approx& \Gamma(0) + \sum\limits_{n=1}^{N_t}\omega_n
\mathcal{V}_{\mathcal{NN}_2}\left((\mathbf{x}_2|_{\Gamma(t_n)},t_n);
\bm{\Theta}_2\right),\, \forall\, 0\leq t\leq T,
\end{array}
\end{equation}
but through a nonlinear iteration process, numerically, since
$\mathcal{V}_{\mathcal{NN}_2}(\mathbf{X}_{2}; \bm{\Theta}_2)$
depends on $\Gamma(t)$ as well. Here we adopt a numerical quadrature
to calculate the time integration in (\ref{interface-update}), with
$\omega_n$, $t_n$ as quadrature weights and points in time,
respectively, for $n=1,\cdots,N_t$, where $t_1=0$, $t_{N_t}=t$. In
fact, we arrange quadrature points $\{t_n\}_1^{N_t}$ in
(\ref{interface-update}) in the ascending order, and choose their
values to be the same as the temporal coordinate of sampling points
which are evolved from their initial positions on $\Gamma(0)$ by
following the fluid trajectory. On the other hand, the distribution
of sampling points on the projection curve of $\Gamma(t)$ onto the
$xy$-plane, $z=t$, preserves the same connectivity at any time $t$.

Before describing such a nonlinear iteration algorithm for tracking
the interface motion through (\ref{interface-update}), we first
elucidate an adaptive strategy to update all weight coefficients of
loss terms, $\omega_{\mathcal{L}_i}$, $\omega_{\mathcal{B}_i}$,
$\omega_{\mathcal{I}_i}$ and $\omega_{\Gamma}$, $i=1,2$, at the
$n$-th epoch during the main training phase of $80{,}000$ epoches.
Denote by $J := \{\mathcal{L}_1, \mathcal{L}_2, \Gamma,
\mathcal{B}_1, \mathcal{B}_2, \mathcal{I}_1, \mathcal{I}_2\}$ the
index set of all loss terms and by $\mathcal{F}_j^n$ the current
discrete loss of type $j \in J$ at the $n$-th epoch, which is
updated as $\bar{\mathcal{F}}_j^{(n)} = (1-\beta)
\bar{\mathcal{F}}_j^{(n-1)} + \beta \mathcal{F}_j^{(n)}$ with $\beta
= 0.2$. Compute the relative scale $r_j^{(n)} =
\bar{\mathcal{F}}_j^{(n)} \big/ \big( \frac{1}{7} \sum_{k \in J}
\bar{\mathcal{F}}_k^{(n)} \big)$, set the raw weight
$\tilde{\omega}_j^{(n)} = 1 \big/ (r_j^{(n)} + \varepsilon)$ with
$\varepsilon = 10^{-6}$, and update $\omega_j^{(n)} = (1-\alpha)\,
\omega_j^{(n-1)} + \alpha\, \tilde{\omega}_j^{(n)}$ with $\alpha =
0.1$, Finally, the weight is set by a normalization $\omega_j^{(n)}
\leftarrow \omega_j^{(n)} \big/ \big(\frac{1}{7}\sum_{k \in J}
\omega_k^{(n)}\big)$ such that $\frac{1}{7} \sum_{k \in J}
\omega_k^{(n)} = 1$.


In addition, the learning rate is also adaptively scheduled to decay
from $1 \times 10^{-3}$ to $1 \times 10^{-6}$ using the
CosineAnnealingLR scheduler~\cite{pytorch_cosineannealinglr} during
the main training phase. Further, to overcome the problem of
insufficient numerical precision caused by the high-contrast jumping
coefficients, we add some observation points on the interface for
the fluid pressure by assigning its exact solution, as done
in Example 2. 
Particularly, we take 5 points on the circular initial interface
$\Gamma(0)$ with angles $\theta \in \{0, 2\pi/5, 4\pi/5, 6\pi/5,
8\pi/5\}$, whose positions at time $t>0$ are then obtained by
$\bx=\bx(0)+\int_0^t \bm v_2(\bx(0),\tau) \, d\tau$, where $\bv_2$
is the exact solution of fluid velocity in the inner subdomain. We
evaluate these 5 points' trajectories at 10 time points, $t \in
\{0.01, 0.12, 0.23, 0.34, 0.45, 0.56, 0.67, 0.78, 0.89, 1.00\}$,
resulting in 50 points with the coordinates $(\bx,t)\in \Gamma(t)$
that lie exactly on the interface driven by the exact solution of
fluid velocity.

Then, a complete numerical algorithm of PINNs approach for solving
two-phase flow problems with a solution-driven moving interface can
be described in Algorithm \ref{alg:solution-dependent-tracking}.
\begin{breakablealgorithm}\label{alg:solution-dependent-tracking}
\caption{PINNs/meshfree method for solving solution-driven moving
interface problems}
\begin{algorithmic}
\State \textbf{[Initialization]:}
\begin{itemize}
    \item Set spatial training points: $\mathbf{x}_{\Gamma,0}^{i}\in\Gamma(0)$, $i = 1, \ldots,
    N_{\Gamma}$, which are arranged in sequence to form a polygonal initial interface in $\mathbb{R}^2$.
    \item Set temporal training points:
    $0=t^{i}_{\Gamma,1}<t^{i}_{\Gamma,2}<\cdots<t^{i}_{\Gamma,N_{T,i}}=T$, where $N_{T,i}\geq
    2$, corresponding to each $\mathbf{x}_{\Gamma,0}^{i}\in\Gamma(0)$, $i = 1, \ldots,
    N_{\Gamma}$. Thus an initial spatiotemporal training set,
    $\bX^0_\Gamma:=\big\{\{\bm{S}_{ij}\}_{j=1}^{N_{T,i}}\big\}_{i=1}^{N_\Gamma}$
    with $\bm{S}_{ij}=(\mathbf{x}_{\Gamma,0}^{i},t^{i}_{\Gamma,j})$, is
    defined on the initial interface 
    $\Gamma(0)\times[0,T]\subset\mathbb{R}^3$ with the number of sampling points
    $M_{\Gamma}=\sum_{i=1}^{N_{\Gamma}}
    N_{T,i}$.
    \item Set sampling parameters, $M_{\mathcal{L}_i}$, $M_{\mathcal{B}_i}$, $M_{\mathcal{I}_i}$,
    $i=1,2$, and generate a spatiotemporal training set,
    $\bX_{\overline{\Omega}\backslash\Gamma}$ in $\overline{\Omega}\backslash\Gamma(0)\times[0,T]\subset\mathbb{R}^3$,
    which is separated by $\Gamma(0)\times[0,T]$.
    \item Set zero-initialized neural networks, $\mathcal{NN}_{1}$ and
    $\mathcal{NN}_{2}$, with $\bm{\Theta}_1=\bm{\Theta}_2=\bm{0}$,
    and set the total number of epoches, $N_{\text{epoch}}$.
\end{itemize}

\vspace{0.1cm} \State \textbf{[The PINNs' training process]:}

\For{$ep=1$ to $N_{\text{epoch}}$}
        \State \textbf{[Update training points on the interface]:}
        \For{each interface point $\mathbf{x}_{\Gamma,0}^{i}\in\Gamma(0)$, $i = 1$ to $N_{\Gamma}$,}
        \For{$k = 2$ to $N_{T,i}$}
        \State $\bullet$ Call the subroutine:
        \State $\ \ $
        \textbf{Interface-Position\big($t^i_{\Gamma,k}$,$N_{\Gamma}$,
        $\{N_{T,i}\}_{i=1}^{N_\Gamma}$,$\bX_\Gamma^{ep-1}$,
$\mathcal{V}_{\mathcal{NN}_2}(\bX_\Gamma^{ep-1};$ $\bm{\Theta}_2)$;
$\{(\mathbf{x}_{\Gamma,k-1}^{i},t^i_{\Gamma,k})\}_{i=1}^{N_\Gamma}
$\big)}.

            \EndFor
        \EndFor
        \State $\bullet$ Update the training set on the interface:
        $\bX_\Gamma^{ep}:=\big\{\{\bm{S}_{ij}\}_{j=1}^{N_{T,i}}\big\}_{i=1}^{N_\Gamma}$,
    where
    $\bm{S}_{ij}=(\mathbf{x}_{\Gamma,j-1}^{i},t^{i}_{\Gamma,j})$.



   \State \textbf{[Generate training sets in subdomains]:}
    \State $\bullet$ Initialize empty subsets of training points,
    $\bX^{ep}_{1} = \bX^{ep}_{2} = \emptyset$.

    \For{each interior training point $(\mathbf{x}^i,t^i)\in \bX_{{\Omega}\backslash\Gamma}$,
    $i=1$ to $M_{\mathcal{L}_1}+M_{\mathcal{L}_2}$,}
        \State $\bullet $ Call the subroutine:
        \State $\ \ $
        \textbf{Interface-Position\big($t^i,N_{\Gamma},\{N_{T,i}\}_{i=1}^{N_\Gamma},\bX_\Gamma^{ep-1},
        \ \mathcal{V}_{\mathcal{NN}_2}(\bX_\Gamma^{ep-1};\bm{\Theta}_2)$; $
\{(\mathbf{x}_\Gamma^{i},t^i)\}_{i=1}^{N_\Gamma} $\big)}.
            \State $\bullet $ Determine if $\mathbf{x}^i$ locates
            inside the polygon formed by
            $\{\mathbf{x}_\Gamma^{i}\}_{i=1}^{N_\Gamma}$ in
            sequence, by means of ``Ray Casting'' algorithm (See Remark
            \ref{rmk:ray-casting}).
            \If{[inside]}
                \State $\bullet $ Add $(\mathbf{x}^i,t^i)$ to
                $\bX_{2}^{ep}$.
            \Else
                \State $\bullet $ Add $(\mathbf{x}^i,t^i)$ to
                $\bX_{1}^{ep}$.
            \EndIf
    \EndFor

    \State $\bullet $ Form the training set: $\bX^{ep}:=\bigcup_{i=1}^2
    \left(\bX_i^{ep}\cup\bX_{\partial\Omega_i\backslash\Gamma}
    \cup\bX_{\Omega_i(0)}\right)\bigcup\bX_\Gamma^{ep}$.
    \State $\bullet $ Update the learning rate: $\eta \leftarrow$
    CosineAnnealingLR($\eta,\ ep$).
    \State $\bullet $ Set the weights of loss terms, $\omega_{\mathcal{L}_i},\omega_{\mathcal{B}_i},
    \omega_{\mathcal{I}_i},\omega_{\Gamma}$, $i=1,2$, as follows:
    \If {$ep$ $\leq 20,000$}
    \State $\bullet $ Set
    $\omega_{\mathcal{L}_i}=\omega_{\mathcal{B}_i}=\omega_{\mathcal{I}_i}
    =\omega_{\Gamma}=1$, $i=1,2$.
    \Else
    \State $\bullet $ Use the adaptive weighting strategy as
    described above for an adaptive setting.
    \EndIf
    \State $\bullet $ Add aforementioned observation data points for the fluid
    pressure, appropriately, if the physical parameters jumps high across the interface.
    \State $\bullet $ Apply the Adam optimizer to minimize
    (\ref{eqn:DNN-interface-method}) or (\ref{newtotalloss}), and
    obtain the current step's predicted solutions,
$\mathcal{V}_{\mathcal{NN}_i}(\bX^{{ep}};\bm{\Theta}_i)$ and
$\mathcal{P}_{\mathcal{NN}_i}(\bX^{{ep}};\bm{\Theta}_i)$, $i=1,2$,
based on the current training set $\bX^{{ep}}$.

\EndFor

\State $\bullet $ \textbf{Return:} The trained networks
$\mathcal{NN}_1$ and $\mathcal{NN}_2$ with the final training set
$\bX=\bX^{ep}$, the finally predicted solution
$\mathcal{V}_{\mathcal{NN}}(\bX;\bm{\Theta}^*)$ and
$\mathcal{P}_{\mathcal{NN}}(\bX;\bm{\Theta}^*)$, and the interface
evolution history along the training epoches,
$\{\bX_\Gamma^{i}\}_{i=0}^{N_{epoch}}$.

\vspace{0.2cm} \State \textbf{[Subroutine for updating interface
position]:} \State
\emph{\textbf{Interface-Position\big($t,N_{\Gamma},\{N_{T,i}\}_{i=1}^{N_\Gamma},\bX_\Gamma,
\mathcal{V}_{\mathcal{NN}}(\bX_\Gamma;\bm{\Theta});\{(\mathbf{x}_\Gamma^{i},t)\}_{i=1}^{N_\Gamma}
$\big)}}
\State $\bullet$ Take spatial training points on the interface at
the initial time from the given training set $\bX_\Gamma$:
$\{(\mathbf{x}_{\Gamma,0}^{i},0)\}_{i=1}^{N_\Gamma}
=\{\bm{S}_{i1}\}_{i=1}^{N_\Gamma}\in\bX_\Gamma$.
        \For{$i = 1$ to $N_{\Gamma}$}
            \State $\bullet$ Take training points corresponding to each spatial
            point $\mathbf{x}_{\Gamma,0}^{i}\in\Gamma(0)$ from the given training set $\bX_\Gamma$:
            $\{(\mathbf{x}^{i}_{\Gamma,j-1},t_{\Gamma,j}^{i})\}_{j=2}^{N_{T,i}}
=\{\bm{S}_{ij}\}_{j=2}^{N_{T,i}}\in\bX_\Gamma$.
            \For{$k = 2$ to $N_{T,i}$}
            \If {$t_{\Gamma,k-1}^{i}< t \leq t_{\Gamma,k}^{i}$}
            \State $\bullet$ Set $n=k$.
            \State $\bullet$ Set $\mathbf{x}^{i}_{\Gamma,t}=\frac{t-t_{\Gamma,k-1}^{i}}
            {t_{\Gamma,k}^{i}-t_{\Gamma,k-1}^{i}}\mathbf{x}^{i}_{\Gamma,k-1}
            +\frac{t_{\Gamma,k}^{i}-t}
            {t_{\Gamma,k}^{i}-t_{\Gamma,k-1}^{i}}\mathbf{x}^{i}_{\Gamma,k-2}$.
            \EndIf
            \EndFor

            \State $\bullet$ Set the displacement, $\mathbf{d}=\mathbf{0}$.
            \For{$k = 1$ to $n-1$}
                \State $\bullet\, $ Evaluate the velocity at
                $t_{\Gamma,k}^{i}$: $\mathbf{v}_k = \mathcal{V}_{\mathcal{NN}}
                (({\mathbf{x}_{\Gamma,k-1}^{i}},t_{\Gamma,k}^{i});\bm{\Theta})$.
                \If {$k=n-1$}
                \State $\bullet\, $ Evaluate the velocity at
                $t$: $\mathbf{v}_{k+1} =\mathcal{V}_{\mathcal{NN}}
                (({\mathbf{x}_{\Gamma,t}^{i}},t);\bm{\Theta})$.
                \State $\bullet\, $ Set $\Delta
                t=t-t_{\Gamma,k}^{i}$.
                \Else
                \State $\bullet\, $ Evaluate the velocity at
                $t_{\Gamma,k+1}^{i}$: $\mathbf{v}_{k+1} = \mathcal{V}_{\mathcal{NN}}
                (({\mathbf{x}_{\Gamma,k}^{i}},t_{\Gamma,k+1}^{i});\bm{\Theta})$.
                \State $\bullet\, $ Set $\Delta
                t=t_{\Gamma,k+1}^{i}-t_{\Gamma,k}^{i}$.
               \EndIf
                \State $\bullet\, $ Update the displacement using the trapezoidal
                rule, $$\mathbf{d} \leftarrow \mathbf{d} + \frac{\Delta t}{2}(\mathbf{v}_k +
\mathbf{v}_{k+1}).$$
            \EndFor
            \State $\bullet$ Update the interface point's position at time $t$:
            $\mathbf{x}^{i}_\Gamma = \mathbf{x}_{\Gamma,0}^{i} + \mathbf{d}$.
        \EndFor
            \State $\bullet\, $ \textbf{Return:} The updated interface training
            points at time $t$: $\{(\mathbf{x}^{i}_\Gamma,t)\}_{i=1}^{N_\Gamma}$.

\end{algorithmic}
\end{breakablealgorithm}

\begin{remark}\label{rmk:ray-casting}
In Algorithm \ref{alg:solution-dependent-tracking}, we employ a
so-called ``Ray Casting'' algorithm to determine
whether a point lies inside or outside a simple
polygon~\cite{preparata1985computational}. From the query point, one
casts a ray (e.g., a horizontal half-line) to infinity and counts
the number of times it crosses the polygon boundary: an odd number
of crossings indicates that the point is inside, while an even
number of crossings means it is outside. This yields a robust and
simple domain classification for a deforming interface represented
by the vertex list $\{\mathbf{x}_{\Gamma}^i\}_{i=1}^{N_\Gamma}$. For
a detailed algorithm, one can refer to \cite[Section
2.2]{preparata1985computational}.
\end{remark}

As elucidated in Algorithm \ref{alg:solution-dependent-tracking}, we
can randomize a training set to train the developed PINNs framework
for the presented two-phase flow problem whose moving interface is
driven by the predicted fluid velocity. However, merely to
investigate the convergence behavior based upon a
training-set-doubling strategy, as done for previous examples, we
uniformly generate the training set for this example as well,
and utilize it in Algorithm \ref{alg:solution-dependent-tracking} to
obtain convergent generalization errors as shown in Table
\ref{tab:NSNS-solution-dependent}, the interface tracking history in
Figure \ref{fig:sampling-points-comparison}, and distributions of
numerical results in Figures \ref{fig:solution-dependent-pinn-10}
and \ref{fig:solution-dependent-error-10}.
\begin{table}[H]
\centering \caption{Loss error and generalization error vs \# of
sampling points in Example 3 with adaptive interface tracking}
\label{tab:NSNS-solution-dependent}
\begin{tabular}{ c c c c c c}
\hline \hline $M_{\mathcal{L}_i}$ & $M_{\mathcal{B}_i}$ &
$M_{\Gamma}$&  $M_{\mathcal{I}_i}$  &  Gen-Error & Loss Error
\\ \hline
$10 \times 10 \times 5$ & $4 \times 4 \times 5$ & $4 \times 5$  & $4 \times 4$  & 4.25e-02  & 3.15e-05 \\
$10 \times 10 \times 5$ & $8 \times 4 \times 5$ & $8 \times 5$  & $8 \times 8$  & 3.67e-02  & 2.88e-05 \\
$10 \times 10 \times 5$ & $16 \times 4 \times 5$ & $16 \times 5$  & $16 \times 16$  & 3.14e-02  & 2.56e-05 \\
$10 \times 10 \times 5$ & $32 \times 4 \times 5$ & $32 \times 5$  &
$32 \times 32$  & 2.72e-02  & 2.24e-05 \\ \hline
$20 \times 20 \times 10$ & $4 \times 4 \times 10$ & $4 \times 10$  & $4 \times 4$  & 3.28e-02  & 2.03e-05 \\
$20 \times 20 \times 10$ & $8 \times 4 \times 10$ & $8 \times 10$  & $8 \times 8$  & 2.76e-02  & 1.85e-05 \\
$20 \times 20 \times 10$ & $16 \times 4 \times 10$ & $16 \times 10$  & $16 \times 16$  & 2.33e-02  & 1.68e-05 \\
$20 \times 20 \times 10$ & $32 \times 4 \times 10$ & $32 \times 10$
& $32 \times 32$  & 1.91e-02  & 1.47e-05 \\ \hline
$40 \times 40 \times 20$ & $4 \times 4 \times 20$ & $4 \times 20$  & $4 \times 4$  & 2.67e-02  & 1.35e-05 \\
$40 \times 40 \times 20$ & $8 \times 4 \times 20$ & $8 \times 20$  & $8 \times 8$  & 2.14e-02  & 1.22e-05 \\
$40 \times 40 \times 20$ & $16 \times 4 \times 20$ & $16 \times 20$  & $16 \times 16$  & 1.78e-02  & 1.08e-05 \\
$40 \times 40 \times 20$ & $32 \times 4 \times 20$ & $32 \times 20$  & $32 \times 32$  & 1.42e-02  & 9.45e-06 \\
\hline \hline
\end{tabular}
\end{table}
From Table \ref{tab:NSNS-solution-dependent} we can see that the
proposed PINNs/meshfree method, together with the developed
interface tracking Algorithm \ref{alg:solution-dependent-tracking},
still achieves comparable accuracy for the solution-driven moving
interface case of two-phase flow problems, in comparison with
previous examples, where the generalization errors and loss errors
are on the same order of $10^{-2}$ and $10^{-5}$, respectively,
demonstrating the effectiveness of the physically consistent method
in interface tracking for complex two-phase flow moving interface
problems. Figure \ref{fig:sampling-points-comparison} illustrates
the history of a solution-driven moving interface tracked along the
training steps in the $(2+1)$-th dimensional spatiotemporal space,
starting with an initial guess of interface evolution at the initial
epoch: a cirular cylinderical interface, and ending with a finally
determined deforming and irregular interface at the final epoch,
which separates the training set to be internal and external subsets
indicated in different colors at each epoch.
\begin{figure}[H]
\centering \subfigure[Initial epoch]{
\includegraphics[width=3.75cm, height=3.5cm]{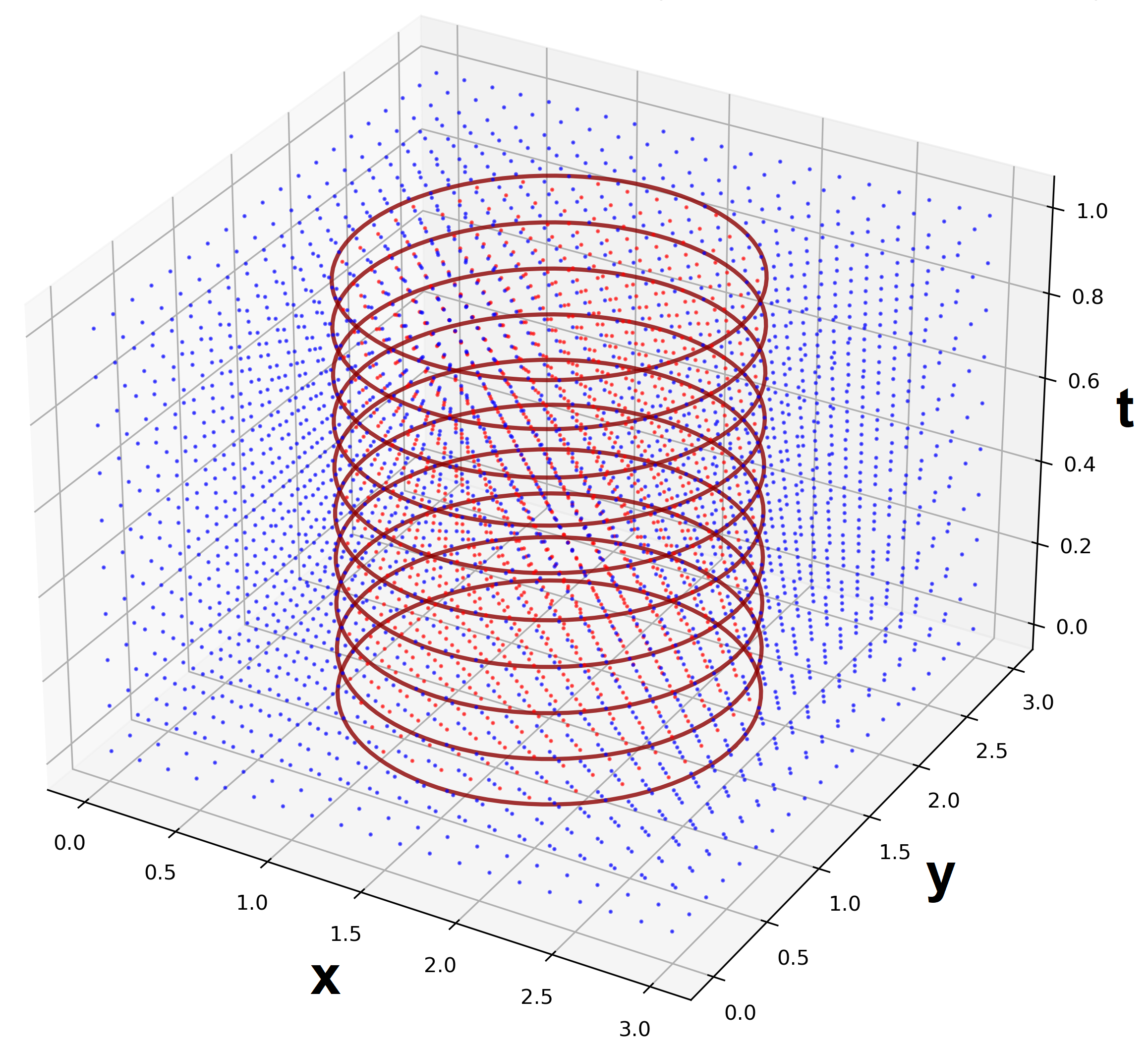}
\label{fig:sampling-points-start}} \subfigure[Middle epoch]{
\includegraphics[width=3.75cm, height=3.5cm]{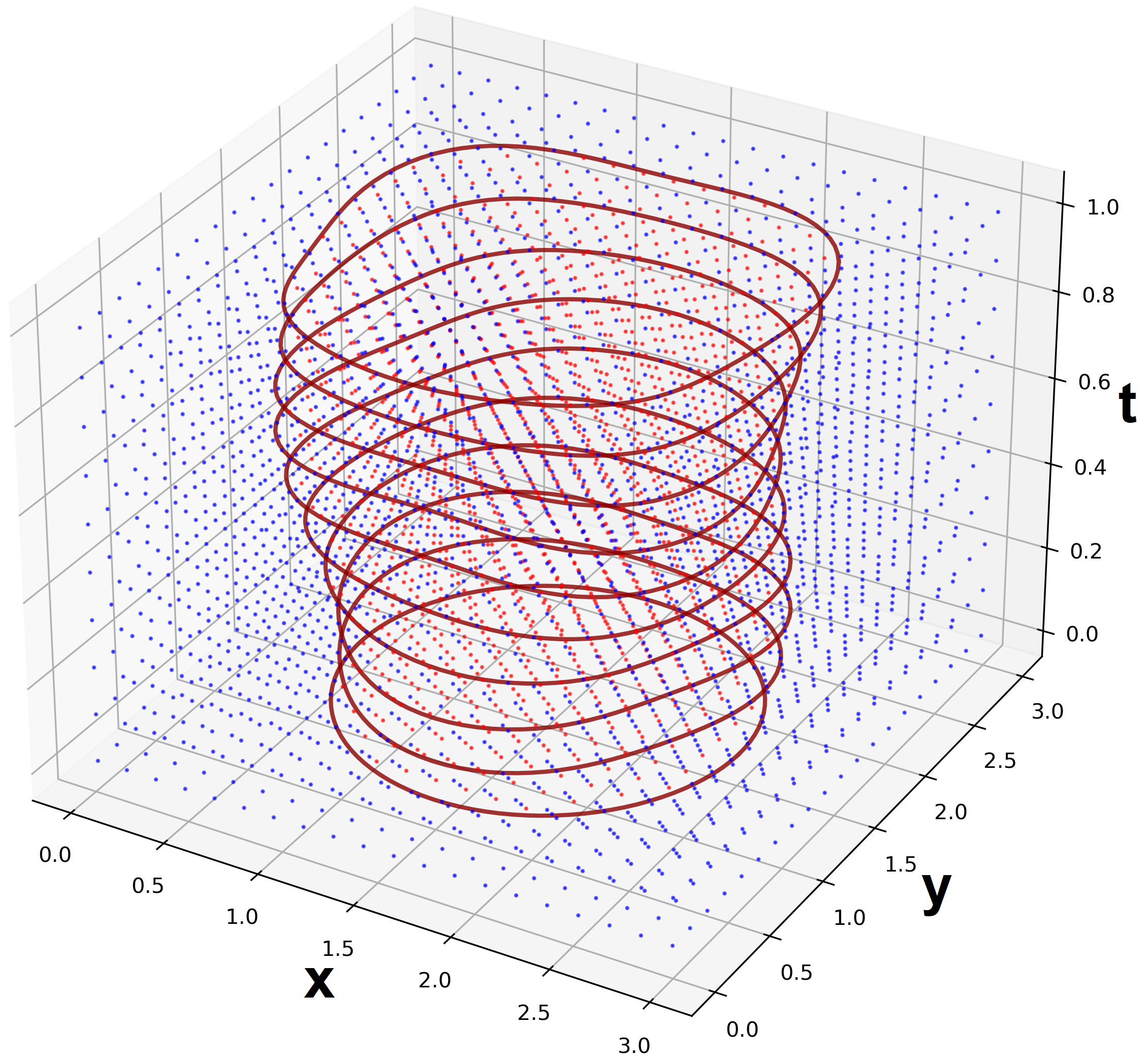}
\label{fig:sampling-points-mid}}\subfigure[Final epoch]{
\includegraphics[width=3.75cm, height=3.5cm]{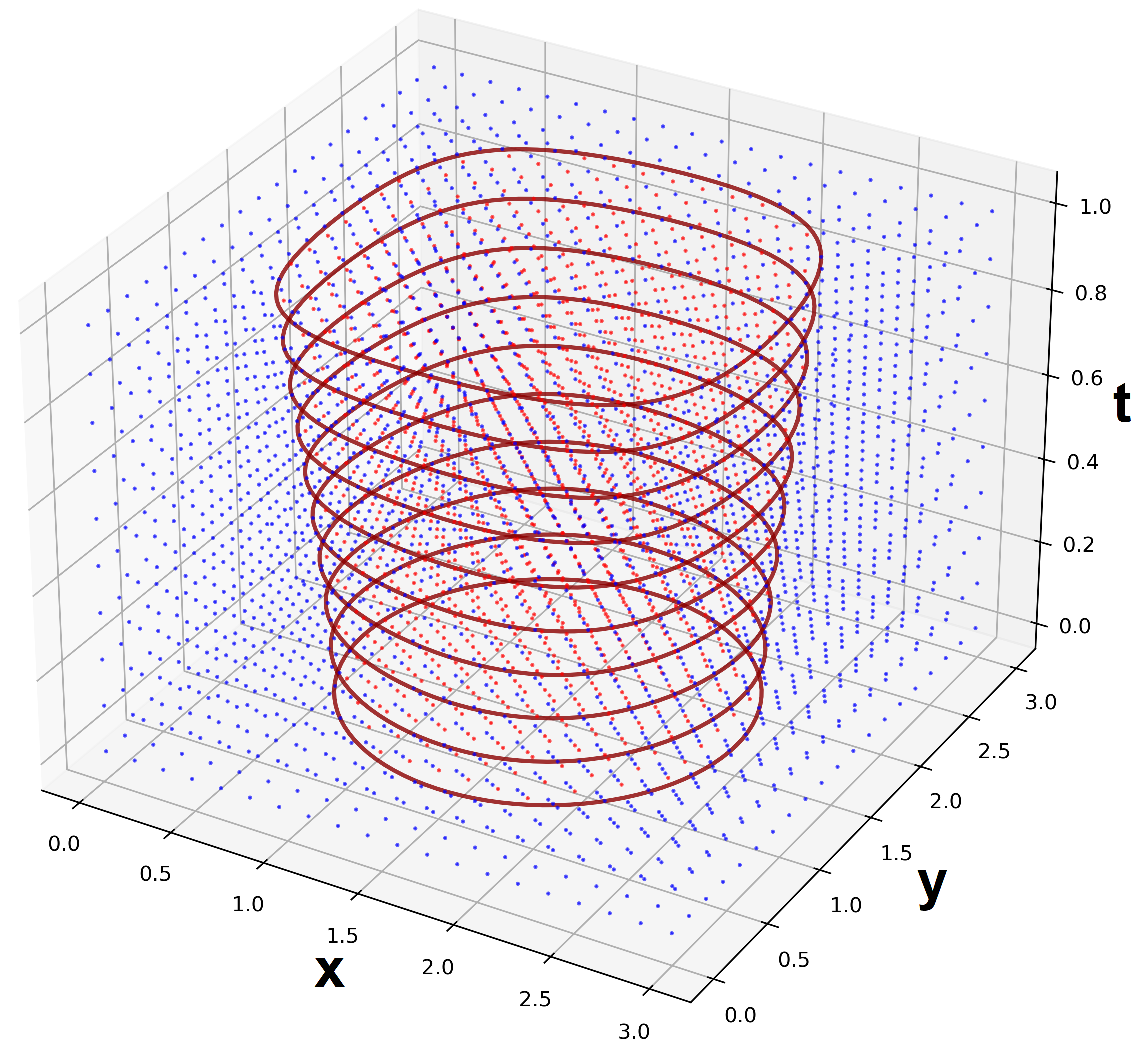}
\label{fig:sampling-points-end} } \caption{Training points
distribution separated by a solution-driven interface tracked along
the training epoches and classified as internal (red) and external
(blue) subsets, where projections of interface onto the $xy$-plane,
$z=t\in\{0.01,\, 0.12,\, 0.23,\, 0.34,\, 0.45,\, 0.56,\, 0.67,\,
0.78,\, 0.89,\, 1.00\}$ are plotted as red curves. \textbf{(a)} In
the beginning of training, the interface is initially set as a
circular cylinder; \textbf{(b)} in the middle of training, the
interface keep deforming due to the solution-driven mechanism by the
PINNs output; \textbf{(c)} at the end of training, the interface is
finally deformed as an irregular shape due to the dependency of
PINNs output.} \label{fig:sampling-points-comparison}
\end{figure}

\begin{figure}[H]
    \centering
\includegraphics[width=12cm, height=6cm]{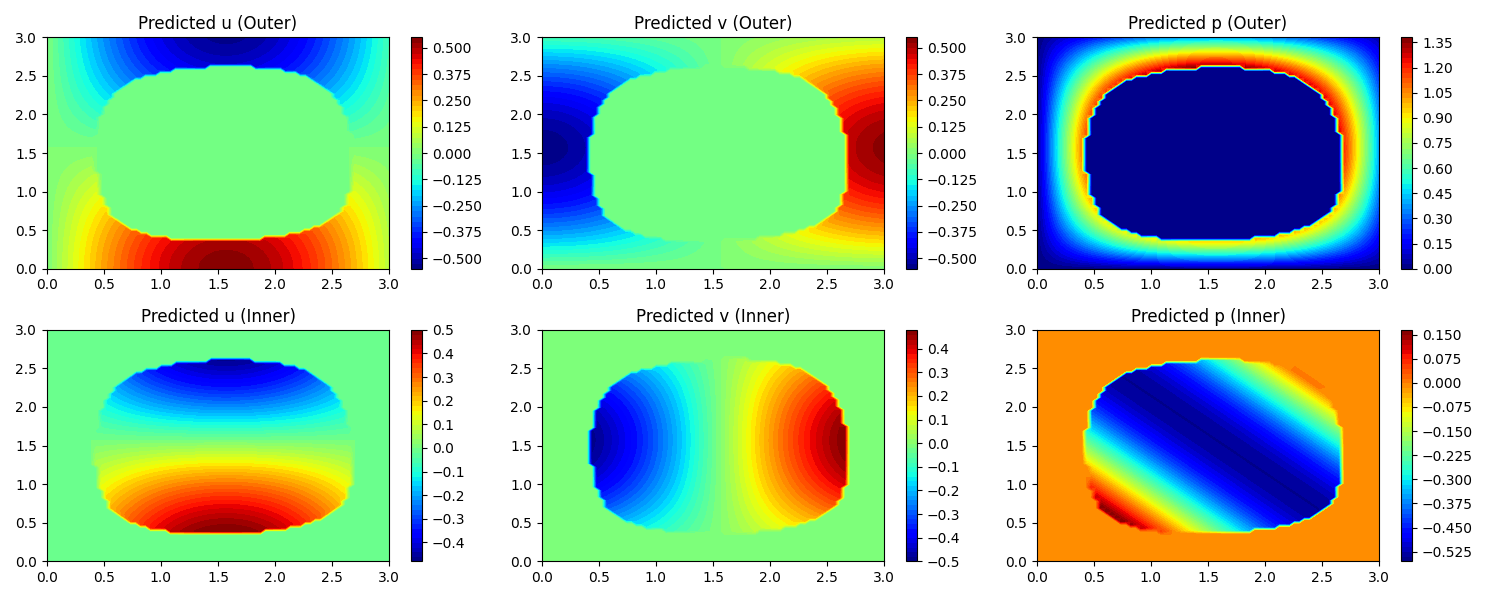}
\caption{PINNs outputs of Example 3 at the terminal time.
\textbf{Left}: The horizontal velocity $u$; \textbf{Middle}: the
vertical velocity $v$; \textbf{Right}: the pressure $p$;
\textbf{Upper}: outside the moving interface; \textbf{Lower}: inside
the moving interface.} \label{fig:solution-dependent-pinn-10}
\end{figure}


\begin{figure}[H]
    \centering
\includegraphics[width=12cm, height=6cm]{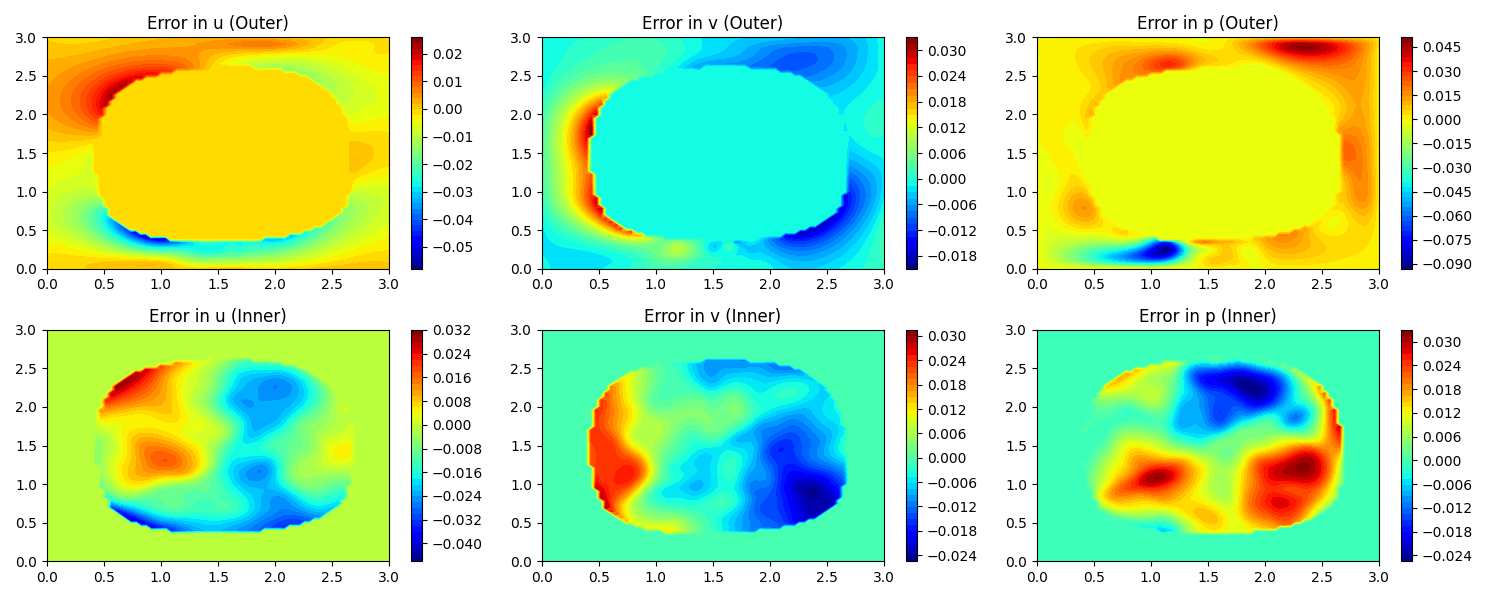}
\caption{Error distributions between PINNs output and exact solution
of Example 3 at the terminal time. \textbf{Left}: The horizontal
velocity $u$; \textbf{Middle}: the vertical velocity $v$;
\textbf{Right}: the pressure $p$; \textbf{Upper}: outside the moving
interface; \textbf{Lower}: inside the moving interface.}
\label{fig:solution-dependent-error-10}
\end{figure}
Figure \ref{fig:solution-dependent-pinn-10} illustrates the PINNs
output of velocity and pressure at the terminal time $t = 1$, with
the distribution of errors from the exact solution shown in Figure
\ref{fig:solution-dependent-error-10}, where the largest errors are
still mainly concentrated near the moving interface on the order of
$10^{-2}$. Overall, the interface tracking approach for the
solution-driven moving interface case strengthens the physical
consistency of the adopted PINNs framework with the underlying
dynamical mechanism of two-phase fluid. The evolution process of
interface deformation within the PINNs' training epoches is
accurately tracked in terms of a nonlinear iteration strategy, which
is especially important for solving complex flow patterns and high
contrast scenarios. In addition, the dynamic regeneration of
training points along the training epoches ensures an accurate
subdomain classification and prevents the interface evolution from
numerical artifacts.

\section{Conclusion}\label{sec:conclusion}
In this work, we develop a comprehensive physics-informed neural
networks (PINNs) framework for solving two-phase flow problems with
moving interfaces. The proposed PINNs/meshfree method addresses
fundamental challenges in computational fluid dynamics by providing
a robust, theoretically grounded approach to handle dynamic
interface problems without requiring any mesh generation or
adaptation. Our key contributions include: (1) the establishment of
a rigorous theoretical foundation by developing a novel approach to
handle moving interface terms in terms of the Reynolds transport
theorem, and delivering comprehensive energy estimates and
convergence guarantees with error bounds that reveal the dominance
of quadrature errors from boundary conditions, initial conditions
and interface conditions, besides the interior PDEs' residual; (2) a
practical guidance of optimal sampling point distribution for
solving two-phase flow moving interface problem, where a
least-squares formulation using piecewise neural network structures
is optimized to approximate velocity and pressure fields in each
fluid subdomain; (3) the development of a practical algorithm to
track the unknown moving interface driven by the solution in an
iterative fashion, which is implemented by a feasible PyTorch-based
solution with a two-phase training strategy to demonstrate
the computational efficiency. 

Our extensive numerical experiments validate the theoretical
framework across three challenging testing examples, including two
cases of the prescribed interface motion and one case of the
solution-driven interface motion, with and/or without high-contrast
physical parameters. Numerical results confirm theoretical
predictions regarding error dominance by quadrature terms and
demonstrate consistent convergence behavior with generalization
errors decreasing as sampling point density increases on critical
boundaries and interfaces. The developed method's meshfree nature,
theoretical rigor, and generalizability make it particularly
attractive for problems with rapidly changing or complex interface
geometries, while its ability to handle high-contrast material
properties and complex geometric transformations positions it as a
valuable tool for real-world engineering applications in
microfluidics, biomedical engineering, and environmental fluid
dynamics. In the future work, we will extend our research directions
to enhance robustness for singular solutions, develop adaptive
sampling strategies, solve higher dimensional problems, explore
multiphysics coupling applications, and etc. Further, as an
advancement in computational fluid dynamics, the developed PINNs
approach may also offer a powerful alternative to traditional
mesh-based methods while maintaining theoretical rigor and practical
applicability for challenging fluid-structure interaction problems
as well.

\section{Acknowledgment}
X. Xie was partially supported by National Natural Science
Foundation of China (12571434) and Open Research Project of National
Key Laboratory of Fundamental Algorithms and Models for Engineering
Simulation. P. Sun was partially supported by a grant from the
Simons Foundation (MPS-TSM-00706640).

\newpage
\bibliographystyle{plain}
\bibliography{DNN4TwoPhase,FSI,DLS}
\end{document}